\documentclass{sig-alternate}

\usepackage[T1]{fontenc}
\usepackage{graphicx}
\usepackage{amssymb}
\usepackage{amsmath}
\usepackage{listings}
\usepackage{algorithm}
\usepackage{algorithmic}
\usepackage{multirow}
\usepackage{float}
\usepackage[caption = false]{subfig}
\DeclareMathOperator*{\argmax}{arg max}
\DeclareMathOperator*{\argmin}{arg min}
\newtheorem{theorem}{Theorem}[section]
\newtheorem{proposition}[theorem]{Proposition}

\begin{document}
%

\conferenceinfo{ODDx$3$'15,} {August 10th, 2015, Sydney, AU.} 
\CopyrightYear{2015} 
\crdata{978-1-4503-3831-8}
\clubpenalty=10000 
\widowpenalty = 10000

\title{Abnormal Subspace Sparse PCA for Anomaly Detection and Interpretation}

\numberofauthors{3}
\author{
\alignauthor
        Xingyan Bin\\
       \affaddr{Tsinghua University}\\
       \affaddr{Beijing, China}\\
       \email{bxy13\\@mails.tsinghau.edu.cn}
\alignauthor
        Ying Zhao\titlenote{Corresponding author}\\
       \affaddr{Tsinghua University}\\
       \affaddr{Beijing, China}\\
       \email{yingz\\@mail.tsinghau.edu.cn}
\alignauthor
        Bilong Shen\\
       \affaddr{Tsinghua University}\\
       \affaddr{Beijing, China}\\
       \email{shenbl12\\@mails.tsinghau.edu.cn}
}
\maketitle

\begin{abstract}
The main shortage of principle component analysis (PCA) based anomaly detection models is their interpretability. In this paper, our goal is to propose an interpretable PCA-based model for anomaly detection and interpretation. The propose ASPCA model constructs principal components with sparse and orthogonal loading vectors to represent the abnormal subspace, and uses them to interpret detected anomalies. Our experiments on a synthetic dataset and two real world datasets showed that the proposed ASPCA models achieved comparable detection accuracies as the PCA model, and can provide interpretations for individual anomalies.

\end{abstract}


\section{Introduction}

Principal Component Analysis (PCA) is one of the best-known statistical analysis techniques for detecting anomalies and has been applied to many kinds of data, such as network intrusion detection, failure detection in production systems, and so on \cite{Anomaly-Survey-09,XuWeiGoogle,Lakhina-2005-sigcomm,jiang2013family}. In these domains, pinpointing the sources of detected anomalies is also very important for real applications such as diagnosing failures and recovering systems/networks. Hence, for each detected anomaly, an ideal model should also be able to interpret the reasons of its detection, which we refer to as the problem of {\bf anomaly interpretation}.

Traditional PCA-based anomaly detection models are not suitable for anomaly interpretation \cite{XuWei-SOSP,PCA-Sensitivity}, as they judge whether a data instance is an anomaly or not based on the length of its projection on the abnormal subspace spanned by the less significant principal components, and there is no direct mapping between PCA's dimensionality-reduced subspace and the original feature space. Existing approaches \cite{XuWei-SOSP} added a separated interpretation step to solve this problem by using techniques such as decision trees. However such indirect interpretation often failed to reveal the true causes of the anomalies detected by PCA-based methods \cite{PCA-Sensitivity}. Another recent work \cite{jiang2013family} proposed the joint sparse PCA (JSPCA) model to identify a low-dimensional approximation of the abnormal subspace, so that {\it all} anomalies can be {\it localized} onto a small subset of original feature variables. However, for individual anomaly interpretation, especially anomalies of different types, we need a more accurate and direct way of interpretation.

This paper aims to design an interpretable PCA-based anomaly detection model. Our key observation is that if we manage to construct principal components (PCs) with {\it sparse} and {\it orthogonal} loading vectors to represent the abnormal subspace, a detected anomaly can be interpreted by identifying the set of such PCs on which the anomaly has large projection values. We propose interpretable {\bf Abnormal subspace sparse PCA (ASPCA)} models for anomaly detection and interpretation in this paper, and make the following two contributions.

First, we formulate two objective functions for ASPCA: one extracts the most significant sparse orthogonal PCs first,  and the other extracts the least significant sparse orthogonal PCs first, which prioritizes the sparsity of the abnormal subspace. To the best of our knowledge,  the proposed ASPCA models are the first PCA-based models that are suitable for individual anomaly detection and interpretation. Second, we propose an optimization method for ASPCA models with a semidefinite programming (SDP) relaxation step and a global sparsity optimization step. Our experiments on a synthetic dataset and two real world datasets showed that the proposed ASPCA models achieved comparable detection accuracies as the PCA model, and can provide interpretations for individual anomalies.  

The rest of this paper is organized as follows. Section~2 introduces the proposed ASPCA models.  Section~3 describes the optimization methods for ASPCA models. Section~4 presents a comprehensive experimental evaluation. Section~5 discusses the related work. Finally, Section~6 provides some concluding remarks.

\section{Related Work}
PCA is mostly known as a dimension reduction tool \cite{jolliffe2002principal}, but it is also widely used as an anomaly detection method \cite{dunia1997multi,Anomaly-Survey-09}.
Wei Xu \emph{et al.} used this technique to analyze logs and detect anomalies on a game server, Hadoop File System (HDFS) \cite{XuWei-SOSP}, and Google's production system \cite{XuWeiGoogle}. Ryohei Fujimaki \emph{et al.} used kernel PCA on the Spacecraft Anomaly Detection Problem \cite{kernelPCA-Space}.
People have implemented this technique on network intrusion detection \cite{Lakhina-2005-sigcomm,Lakhina-2004-sigcomm,jiang2013family,PCA-KDD99-2006}.
Anukool Lakhina \emph{et al.} applied this technique on the problem of network flood monitoring using PCA on the matrix of time and Origin-Destination pairs \cite{Lakhina-2005-sigcomm,Lakhina-2004-sigcomm}. Firstly, they used the volume of communication \cite{Lakhina-2004-sigcomm} in their model, and then they extent their model with the entropy of communication volumes and applied PCA with multiple subspace \cite{Lakhina-2005-sigcomm}. Ling Huang \emph{et al.} tried to design an online PCA-based detection method for scalability and communication efficiency \cite{huang2006network,INFOCOM-Distributed-PCA}.

One of the major disadvantages of PCA as a dimension reduction tool is its poor interpretability. Ian Jolliffe \emph{et al.} introduced the concept of sparse PCA which adds a constraint on the sparsity of loading vectors \cite{SPCA-2003}. Since, various methods solving the sparse PCA problem were proposed in the literature, for example \cite{SPCA-2006} and \cite{SPCA-SDP}. Hui Zou \emph{et al.} transformed the sparse PCA problem to a regression-type problem with an elastic net regularization, which could be solved by an alternating minimization scheme \cite{SPCA-2006}. Alexandre d'Aspremont \emph{et al.} proposed a semi-definite programming (SDP) relaxation to the sparse PCA optimization problem \cite{SPCA-SDP}.

PCA-based anomaly detection methods also suffer from the shortage of interpretability \cite{PCA-Sensitivity,XuWei-SOSP}. Ruoyi Jiang \emph{et al.} introduced the joint sparse PCA method for anomaly localization inspired by sparse PCA \cite{jiang2011JSPCA,jiang2013family} , and they followed the alternating minimization framework \cite{SPCA-2006} to solve the optimization problem \cite{jiang2013family}. Wei Xu \emph{et al.} also tried to interpret the results returned by a PCA anomaly detection model with decision trees trained by the data labeled by the PCA model \cite{XuWei-SOSP}, which as shown in our experimental results, can be misleading and fail to reveal the true reason behind the PCA model.
\section{PCA for Anomaly Detection and Interpretation}
\subsection{Notations}
Bold uppercase letters such as ${\bf X}$ denote a matrix and bold lowercase letters such as ${\bf x}$ denote a column vector. Greek letters such as $\lambda,\mu$ are coefficients. $||{\bf X}||_F$ is the Frobenius norm of ${\bf X}$, and $||{\bf X}||_{1,1}$ is the $L_{1,1}$ norm of ${\bf X}$ as $||{\bf X}||_{1,1} = {\bf 1}|{\bf X}|{\bf 1}^T$.

A dataset is represented as an $n \times p$ data matrix ${\bf D}$, where each row vector corresponds to a $p$-dimensional data instance, and each column vector corresponds to a feature variable. ${\bf A} = {\bf D}^T{\bf D}$ is ${\bf D}$'s covariance matrix. $Tr({\bf A})$ represents the trace of matrix ${\bf A}$. $Card({\bf A})$ denotes the cardinality (number of non-zero elements) of matrix ${\bf A}$. ${\bf I}$ is the identity matrix. $\mathbb{S}^p$ is the set of all symmetric semidefinite matrices in $\mathbb{R}^{p \times p}$.

\subsection{PCA for Anomaly Detection}
Principal Component Analysis (PCA) is a dimensionality-reduction technique that captures the highest variance of a multi-dimensional dataset in a lower dimensional subspace defined by a set of orthogonal eigen vectors.  Given a $p$-dimensional dataset, a detection model can be constructed by forming a ``normal subspace" (defined by the first $k$ principal components returned by PCA) and an ``abnormal subspace" (the remaining subspace by removing the normal subspace). Since the normal subspace captures the highest variance of the dataset, PCA-based detection methods assume that this $k$-subspace corresponds to the normal trends of the dataset, and all normal data tends to have almost zero length projection on the abnormal subspace. Therefore, given a $p$-dimensional data, the model can detect whether it is a anomaly or not based on whether it is primarily expressed by the normal or abnormal subspace \cite{PCA-Sensitivity}.

More formally, let ${\bf V}_1=({\bf v}_1,\cdots,{\bf v}_k)$ be the normal subspace defined by the first $k$ principal components with ${\bf v}_1,\cdots,$ ${\bf v}_k$ being the orthogonal loading vectors, and ${\bf V}_2=({\bf v}_{k+1}, \cdots,$ ${\bf v}_p)$ be the abnormal subspace defined by the remaining $p-k$ principal components with ${\bf v}_{k+1},\cdots,{\bf v}_p$ being the orthogonal loading vectors of these PCs. Given a $p$-dimensional data $y$, its residual $\hat{{\bf y}}$ is defined as:

\begin{equation}
\hat{{\bf y}} = {\bf y}-{\bf V}_1{\bf V}_1^T{\bf y}.
\end{equation}

The squared length of ${\bf \hat{y}}$, called the squared prediction error (SPE),  is the metric to indicate whether ${\bf y}$ is an anomaly or not. The larger SPE is, the more likely ${\bf y}$ is an anomaly.

\subsection{Anomaly Interpretation}

When the SPE score of a given instance ${\bf y}$ is over a predefined threshold, ${\bf y}$ is considered as an anomaly. It is then important to understand where the abnormality of ${\bf y}$ comes from, {\it i.e.}, what anomalous feature behaviors of ${\bf y}$ are more responsible for distinguishing ${\bf y}$ from normal data. We call this problem as {\bf Anomaly Interpretation}. The anomaly interpretation for PCA is difficult, as there is no direct mapping between PCA's dimensionality-reduced subspace and the original feature space for anomaly \cite{PCA-Sensitivity}. In other words, the length of $\hat{{\bf y}}$ can be used to detect anomaly, whereas interpreting $\hat{{\bf y}}$ directly is meaningless. 

Given the normal subspace ${\bf V}_1$ and abnormal subspace ${\bf V}_2$, we can rewrite $\hat{{\bf y}}$ as follows:
\begin{equation}
\hat{\bf{y}} = {\bf y}-{\bf V}_1{\bf V}_1^T{\bf y} = {\bf V}_2{\bf V}_2^T{\bf y}.
\end{equation}
 
To design an interpretable PCA-based anomaly detection model, we have the following proposition. 

\begin{proposition}
Given ${\bf V}_2=({\bf v}_{k+1}, \cdots, {\bf v}_p)$, where \\ ${\bf v}_{k+1},\cdots,{\bf v}_p$ are orthogonal loading vectors, SPE can be expressed by
\begin{equation}
SPE= \hat{{\bf y}}^T\hat{{\bf y}} = \sum_{i=k+1}^{p}({\bf v}_i^T {\bf y})^2.
\end{equation}
\end{proposition}

\begin{proof}
\begin{equation}
\label{equation:SPE}
\begin{split}
SPE &= \hat{{\bf y}}^T\hat{{\bf y}} = {\bf y}^T{\bf V}_2 {\bf V}_2^T {\bf V}_2 {\bf V}_2^T {\bf y}\\
       & = ({\bf y}^T{\bf V}_2)({\bf V}_2^T {\bf V}_2)({\bf V}_2^T {\bf y})= ({\bf V}_2^T {\bf y})^T({\bf V}_2^T {\bf y})\\
       & = \sum_{i=k+1}^{p}({\bf v}_i^T {\bf y})^2.
\end{split}
\end{equation}
\end{proof}

\begin{figure}
	\centering
	\subfloat[Normal Instances of the Synthetic Data]
	{
		\includegraphics[width=80mm]{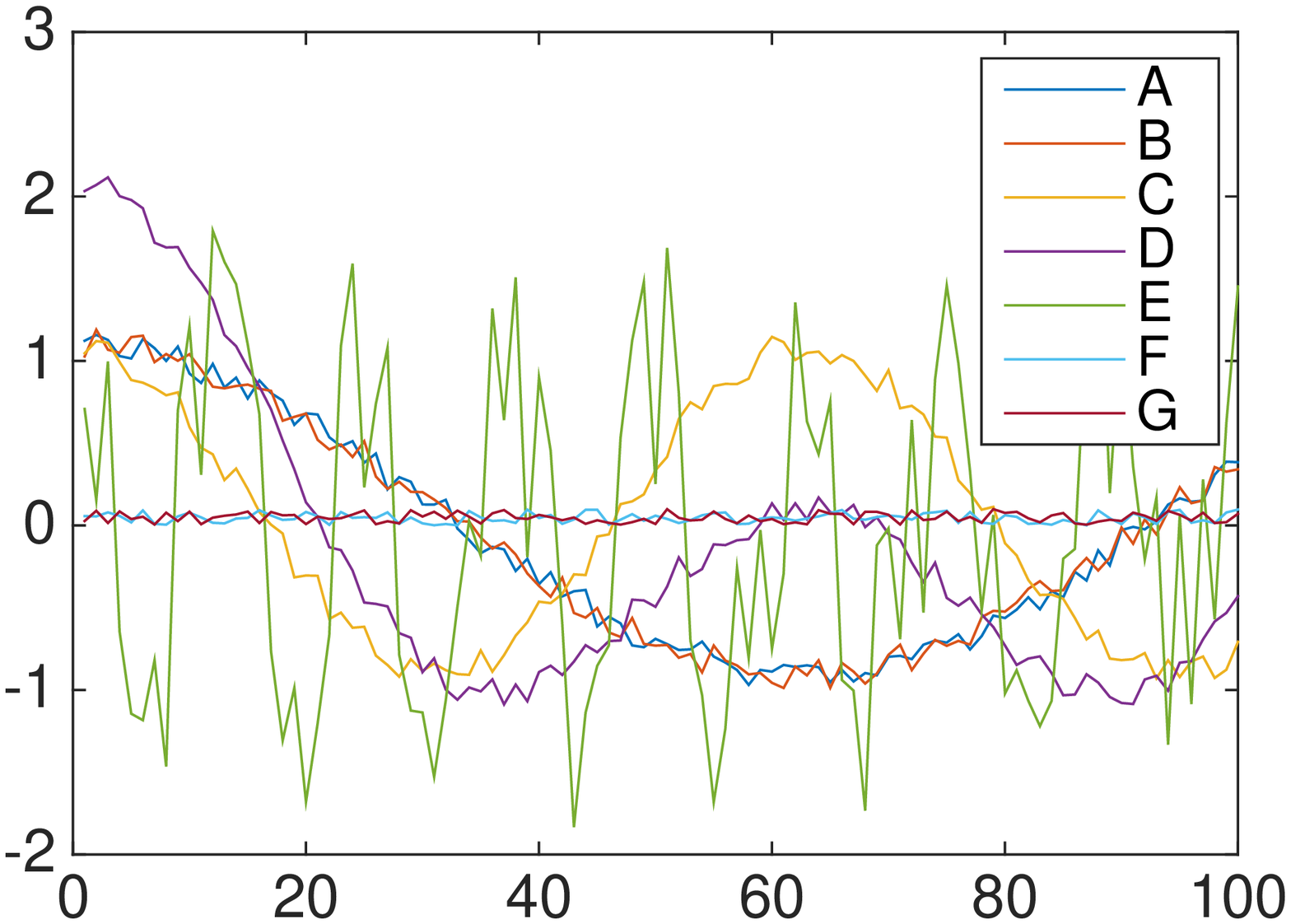}
		\label{fig:synthesized_data}
	}
    \\
	\subfloat[PCA]{
		\includegraphics[height=30mm]{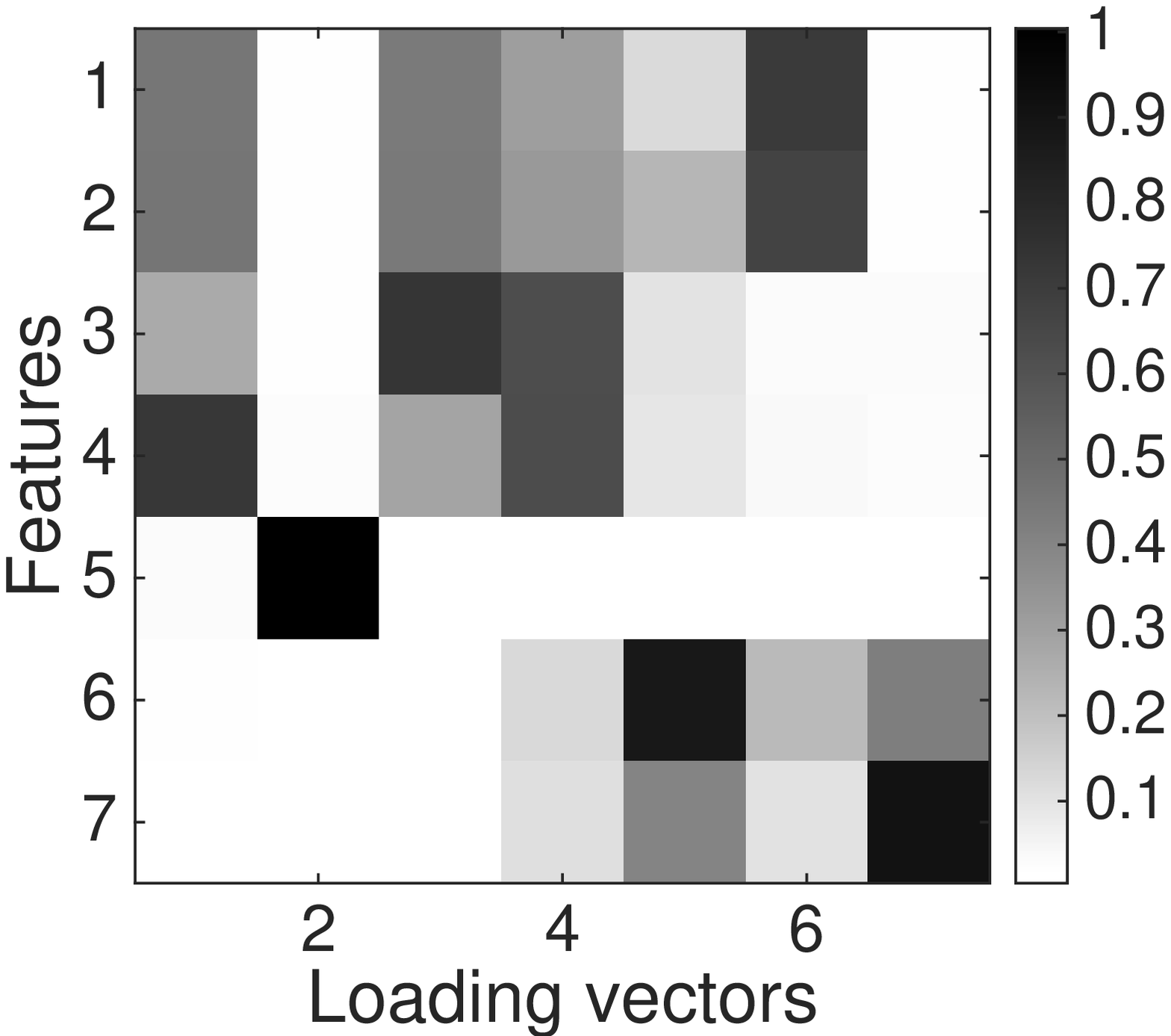}
		\label{fig:pcs:PCA}
	}
	\subfloat[ASPCA]{
		\includegraphics[height=30mm]{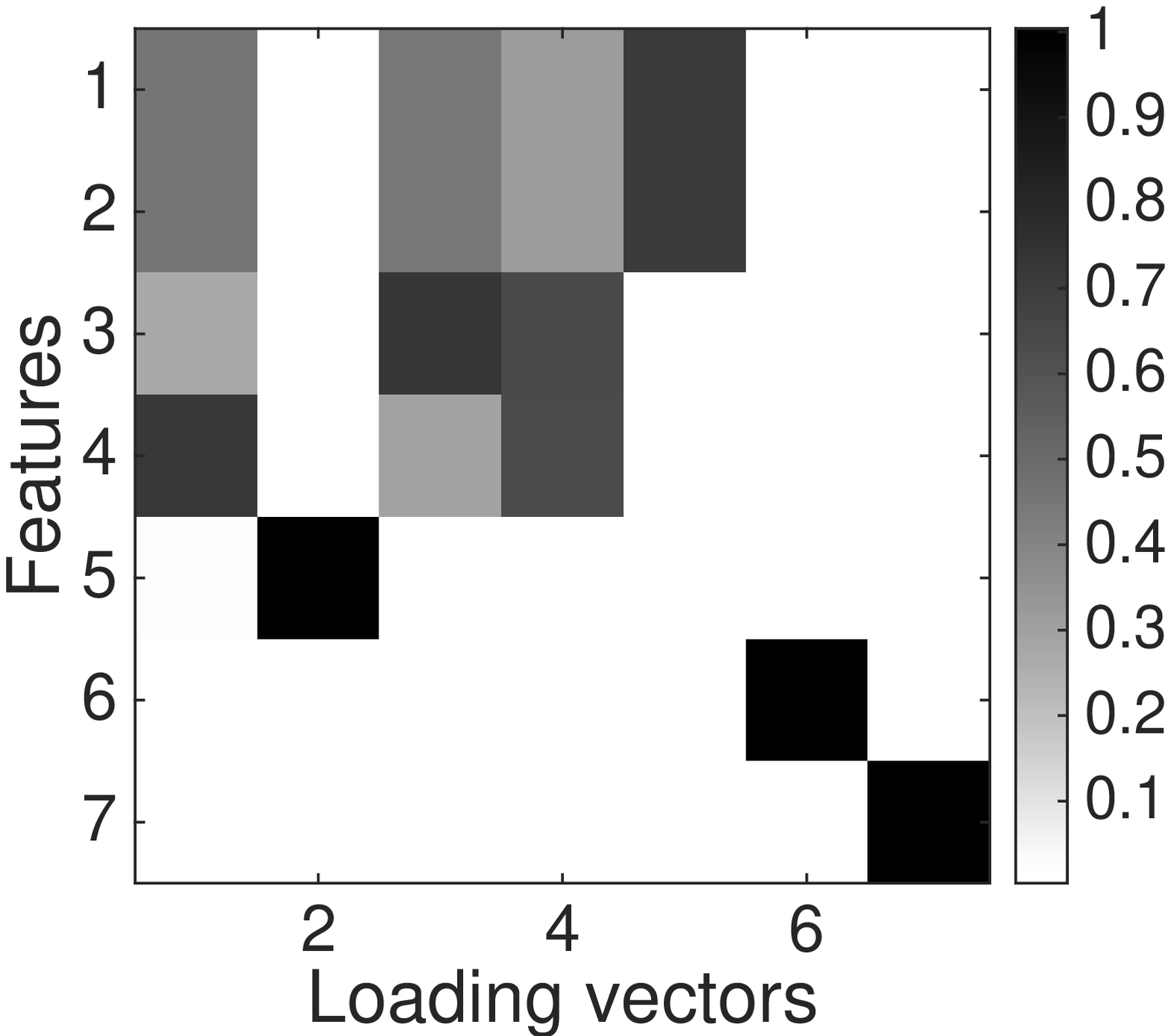}
		 \label{fig:pcs:ASPCA}
	}
\caption{Synthetic data and loading matrices obtained by PCA and ASPCA}
\label{fig:mappingmatrix}
\end{figure}

In other words, SPE is equal to the square sum of ${\bf y}$' scalar projection on each abnormal PCs, so that we can identify the set of PCs that are responsible for the abnormality indicated by high projection values. Unfortunately, these PCs are still difficult to interpret, since each abnormal PC is complicated as it is a linear combination of all feature variables. To make them interpretable, we have to make these abnormal PCs sparse, \emph{i.e.}, each represented by a few feature variables. Hence, our key observation is that if we manage to extract PCs with {\bf sparse} and {\bf orthogonal} loading vectors to represent abnormal subspace, these loading vectors can be used to detect and interpret anomalies. The orthogonality guarantees that Eqn.~\ref{equation:SPE} holds, so that the abnormality can be translated to high projection values on a set of abnormal PCs, while the sparsity guarantees that these abnormal PCs are interpretable.  We call the above method as the {\bf Abnormal Subspace Sparse PCA (ASPCA)} method. Now we use an example to illustrate this idea.

We synthesized a dataset with 500 normal records and 15 anomalies (first 100 normal records are shown in Figure~\ref{fig:synthesized_data}). Each data record has 7 features named from $A$ to $G$, and the normal records were generated with four patterns, $A \approx B$, $D \approx C + A$, $F \approx 0$, and $G \approx 0$.  The anomalies were generated as three categories by breaking the first three patterns, respectively. The loading matrix of PCs obtained by PCA is shown in Figure~\ref{fig:pcs:PCA}, where the last four PCs can be used to detect anomalies but difficult to interpret. Now, if we can make the loading vectors of last four PCs {\bf sparse} and {\bf orthogonal} as shown in Figure~\ref{fig:pcs:ASPCA}, they can be used to detect and interpret anomalies simultaneously. 
Now, the interpretation of a detected anomaly can be conducted in two steps. First, we can identify the set of projections that contribute the most for a high SPE score according to Eqn.~\ref{equation:SPE}. Then, we can interpret these projections one by one, by identifying which original feature variables are responsible for each projection, and how each projection triggers a high SPE score.

\begin{figure}
	\centering
	\includegraphics[height=30mm]{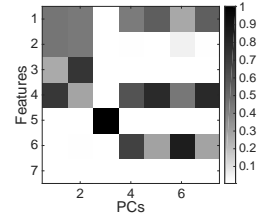}
\caption{Loading matrix obtained by JSPCA}
\label{fig:loading matrix:synthetic}
\end{figure}

Recently, Jiang {\it et al.} \cite{jiang2013family} proposed a joint sparse PCA (JSPCA) model to achieve a sparse representation of the abnormal subspace too. The main idea was to identify a low-dimensional approximation of the abnormal subspace using a subset of feature variables, where all abnormal PCs are represented by the same subset of feature variables as shown in Figure~\ref{fig:loading matrix:synthetic}. Although JSPCA can identify the set of features that distinguish the anomalies, it has two limitations that fail to meet our goals. First, the features identified by JSPCA are optimized for all anomalies as a whole. In particular, if anomalies are of different types, which is a common case for domains such as network intrusion detection or system failure detection, they should be interpreted by different sets of features inherently. As an unsupervised method, JSPCA cannot assume that anomalies in the dataset are of the same type, and cannot handle them well if they indeed are of different types. Second, JSPCA can only identify the important features for anomaly detection, but no direct interpretation as why anomalies are detected.  

\subsection{Abnormal Subspace Sparse PCA}

Now we need to formulate the objective function of the Abnormal Subspace Sparse PCA (ASPCA) problem. The recently studied sparse PCA framework \cite{SPCA-SDP} adds a sparsity constraint on the principal components (PCs). However, we cannot use this framework directly to solve our problem. The main reason is that the sparse PCA framework usually does not enforce orthogonality on the resultant sparse PCs.  Consequently, the resultant sparse PCs cannot be used to define the normal and abnormal subspaces, as the abnormal PCs are not the orthogonal complement of the normal PCs. 

By enforcing orthogonality, sparse PCA can be used to solve our ASPCA problem, which we denote as {\it forward} ASPCA (shorted as ASPCA-F).
Given a covariance matrix ${\bf A}$ and a sparsity constraint constant $k$, for each $i = 1,...,p$, ASPCA-F tries to solve:
\begin{equation}
\label{equation:typical_MSPCA-F}
\begin{split}
& \argmax_{{\bf v}_i}\ {\bf v}_i^T{\bf A}{\bf v}_i\\
& s.t.\ {\bf v}_i^T{\bf v}_i=1, \ {\bf v}_i^T{\bf v}_j=0\ \forall 1\leq j<i, Card({\bf v}_i)\leq k.
 \end{split}
\end{equation}

The last $d$ loading vectors obtained by solving Eqn.~\ref{equation:typical_MSPCA-F} are used for detecting and interpreting anomalies. 

One of the drawbacks of the ASPCA-F framework is that the last abnormal PCs tend to have poor sparsity.  To sovle this problem, we propose a {\it Backward} ASPCA framework (shorted as ASPCA-B) that extracts the least significant orthogonal PCs first, which prioritizes the optimization of the sparsity of the abnormal subspace. To see how it works, we first show that the process of standard PCA can be reversed, so that eigen vectors with smaller eigen values are extracted first by the following proposition.

\begin{proposition}
 Given a covariance matrix ${\bf A} = {\bf D}^T{\bf D}$, if we have already extracted the eigen vectors ${\bf v}_{k+1},{\bf v}_{k+2},....,{\bf v}_n$ with the $n-k-1$ smallest eigen values, and remaining eigen vectors are ${\bf v}_1, {\bf v}_2,....,{\bf v}_k$ with eigen values $\lambda_1 \geq \lambda_2 \geq ... \geq \lambda_k$ of $A$, the solution of Eqn. \ref{equation:proposition:init} is the eigen vector with the eigen value $\lambda_k$.
\begin{equation}
\begin{split}
 & \argmin_{{\bf v}}\ {\bf v}^T{\bf A}{\bf v}\\
 & s.t. {\bf v}^T{\bf v} =1,\ {\bf v}^T{\bf v}_i=0\ \forall k<i\leq n
\end{split}
\label{equation:proposition:init}
\end{equation}
\end{proposition}
\begin{proof}
We project ${\bf v}$ on $({\bf v}_1,...,{\bf v}_n)$, ${\bf v} = \sum_{i=1}^{n} \alpha_i {\bf v}_i = \sum_{i=1}^{k} \alpha_i {\bf v}_i$, where $\alpha_i = {\bf v}^T{\bf v}_i$. 
As ${\bf v}_i^T{\bf v}_j=0, i \neq j$,  we have ${\bf v}^T{\bf v} = \sum_{i=1}^{k} \alpha_i^2 = 1$. Then,
\begin{equation*}
\begin{split}
 & {\bf v}^T{\bf A}{\bf v}\\
 =& (\sum_{i=1}^{k} \alpha_i {\bf v_i})(\sum_{i=1}^{k} \alpha_i {\bf A}{\bf v}_i)
 =(\sum_{i=1}^{k} \alpha_i {\bf v}_i)(\sum_{i=1}^{k} \alpha_i \lambda_i  {\bf v}_i)\\
 =& \sum_{i=1}^{k} \alpha_i^2 \lambda_i {\bf v}_i^T {\bf v}_i + \sum_{i=1}^{k} \sum_{j=1,j \neq i}^{k} \alpha_i \alpha_j \lambda_j {\bf v}_i^T {\bf v}_j\\
 =& \sum_{i=1}^{k} \alpha_i^2 \lambda_i\\
\end{split}
\end{equation*}

As $\lambda_i \geq \lambda_k, i\leq k$, we know $ {\bf v}^T{\bf A}{\bf v} = \sum_{i=1}^{k} \alpha_i^2 \lambda_i \geq \sum_{i=1}^{k} \alpha_i^2 \lambda_k = \lambda_k $. And we know $\min_{\bf v}{{\bf v}^T{\bf A}{\bf v}} \leq {\bf v}_k^T{\bf A}{\bf v}_k = \lambda_k$, so $\min_{\bf v}{{\bf v}^T{\bf A}{\bf v}} = \lambda_k$.

With the optimum $\hat{{\bf v}} = \sum_{i=1}^{k} \hat{\alpha}_i {\bf v}_i$, we have:
\begin{equation*}
\begin{split}
\lambda_k - \sum_{i=1}^{k} \hat{\alpha}_i^2 \lambda_i = \sum_{i=1}^{k} \hat{\alpha}_i^2 (\lambda_k - \lambda_i) = 0 \\
\lambda_k - \lambda_i\leq 0, i<k
\end{split}
\end{equation*}
  So, we know that $\hat{\alpha}_i \neq 0$ only if $\lambda_i = \lambda_k$. Hence, $\hat{{\bf v}}$ is a linear combination of the eigen vectors with eigen value $\lambda_k$, and $\hat{{\bf v}}$  must be an eigen vector with eigen value $\lambda_k$ too.

\end{proof}

Obviously, the proposition also holds for $k=n$. Together we see that using Eqn.~\ref{equation:proposition:init}, eigen vectors can be calculated in an increasing order of eigen values. Now, we add a sparsity constraint to Eqn.~\ref{equation:proposition:init} and form the objective function for our ASPCA-B framework as follows.

Given a covariance matrix ${\bf A}$ and a sparsity constraint constant $k$, for each $i = 1,...,d$, our ASPCA-B framework tries to solve:
\begin{equation}
\label{equation:typical_MSPCA-B}
\begin{split}
& \argmin_{{\bf v}_i}\ {\bf v}_i^T{\bf A}{\bf v}_i\\
& s.t.\ {\bf v}_i^T{\bf v}_i=1, \ {\bf v}_i^T{\bf v}_j=0\ \forall 1\leq j<i, Card({\bf v}_i)\leq k.
 \end{split}
\end{equation}

When we extract $d$ loading vectors $\bf{v}_1...\bf{v}_d$ to span a subspace ${\bf S}_a$,  we make sure that the orthogonal complement of ${\bf S}_a$ has major variance for describing the normal patterns in the dataset, so that ${\bf S}_a$ is the abnormal subspace and the resultant $d$ sparse principal components can be used to detect and interpret anomalies.

\section{Methodology}

We derive a solution for Eqn.~\ref{equation:typical_MSPCA-F} following the semidefinite programming (SDP) relaxation framework proposed by \cite{SPCA-SDP}. We then modify it to solve Eqn.~\ref{equation:typical_MSPCA-B}. Finally,we further optimize the sparsity of all the obtained abnormal components with the constraint of spanning the same subspace using the alternating minimization scheme inspired by \cite{SPCA-2006}. 

\paragraph{\bf Solving ASPCA-F with SDP Relaxation}
We first transform Eqn. \ref{equation:typical_MSPCA-F} without the orthogonality constraint ${\bf v}_j^T {\bf v}_i=0, \forall 1\leq j<i$ to Eqn.\ref{equation:aspca_sdp} through a SDP relaxation.

\begin{equation}
\label{equation:aspca_sdp}
\begin{split}
 & \argmax_{{\bf X}_i \in \mathbb{S}^p}\ Tr({\bf A}{\bf X}_i)\\
 & s.t.\ {\bf X}_i\succeq0, rank({\bf X}_i)=1, \ Tr({\bf X}_i)=1, Card({\bf X}_i) < k^2\\
\end{split}
\end{equation}

where ${\bf X}_i$ is a positive semi-definitive matrix with the constraint $rank({\bf X}_i)=1$, which can be uniquely decomposed as ${\bf X}_i={\bf v}_i{\bf v}_i^T$. With ${\bf X}_i={\bf v}_i{\bf v}_i^T$, $Tr({\bf X}_i)=1$ is equivalent to ${\bf v}_i^T{\bf v}_i=1$, $Card({\bf X}_i)\leq k^2$ is equivalent to $Card({\bf v}_i)\leq k$, and we have ${\bf v}_i^T{\bf A}{\bf v}_i = Tr({\bf A}({\bf v}_i{\bf v}_i^T))=Tr({\bf A}{\bf X}_i)$.

Now let ${\bf V}_i = ({\bf v}_1, {\bf v}_2, ... ,{\bf v}_i)$ and ${\bf R}_i = {\bf V}_i{\bf V}_i^T$, the orthogonality constraint ${\bf v}_j^T {\bf v}_i=0, \forall 1\leq j<i$ is equivalent to $||{\bf V}_{i-1}^T{\bf v}_i||_2^2 = 0 $, and $ ||{\bf V}_{i-1}^T{\bf v}_i||_2^2 = {\bf v}_i^T{\bf V}_{i-1}{\bf V}_{i-1}^T{\bf v}_i=Tr({\bf R}_{i-1}{\bf X}_i)=0$. Similarly as in \cite{SPCA-SDP}, we relax $Card({\bf X}_i) < k^2$ to $||{\bf X}_i||_{1,1}<k$ and move it to the objective function with a coefficient $\lambda$. Finally, the non-convex constraint $rank({\bf X}_i)=1$ is dropped, and we have an objective function that can be solved by semidefinite programming (SDP) as in Eqn. \ref{equation:aspca_f_sdp_final}.

\begin{equation}
\label{equation:aspca_f_sdp_final}
\begin{split}
& \argmax_{{\bf X}_i \in \mathbb{S}^p}\ Tr({\bf A}{\bf X}_i)-\lambda ||{\bf X}_i||_{1,1}\\
& s.t.\ {\bf X}_i\succeq0,\ Tr({\bf X}_i)=1, Tr({\bf R}_{i-1}{\bf X}_i)=0\\
\end{split}
\end{equation}

As $rank({\bf X}_i)$ might not be $1$, so we use the dominant eigenvector of ${\bf X}_i$ as the approximate solution for ${\bf v}_i$.

\paragraph{\bf Solving ASPCA-B with SDP Relaxation}

To solve Eqn.~\ref{equation:typical_MSPCA-F}, following the same steps above, we can get Eqn.~\ref{equation:aspca_b_sdp_final}, which is still a convex programming problem and can be solved by semidefinite programming (SDP).
\begin{equation}
\label{equation:aspca_b_sdp_final}
\begin{split}
& \argmin_{{\bf X}_i \in \mathbb{S}^p}\ Tr({\bf A}{\bf X}_i)+\lambda ||{\bf X}_i||_{1,1}\\
& s.t.\ {\bf X}_i\succeq0,\ Tr({\bf X}_i)=1, Tr({\bf R}_{i-1}{\bf X}_i)=0\\
\end{split}
\end{equation}

\paragraph{\bf Global Sparsity Optimization}
Let ${\bf V}=({\bf v}_1,...,{\bf v}_d)$ be the set of sparse loading vectors extracted by solving 
Eqn.~\ref{equation:aspca_f_sdp_final} or Eqn.~\ref{equation:aspca_b_sdp_final}, which is also a set of basis vectors spanning the abnormal subspace. Notice that for any set of basis vectors ${{\bf c}_1,...,{\bf c}_d}$ spanning the same subspace, we have
\begin{equation}
SPE= \sum_{i=1}^{d}({\bf v}_i^T {\bf y})^2 = \sum_{i=1}^{d}({\bf c}_i^T {\bf y})^2
\end{equation}
Hence, we can employ a global sparsity optimization step to make the basis vectors of the same abnormal subspace sparser. To this end, we form the following optimization problem on an orthogonal transformation matrix $X$, 

\begin{equation}
\begin{split}
&\argmin_{{\bf X}}\ ||{\bf V}{\bf X}||_{1,1}\\
& s.t.\ {\bf X}^T{\bf X}=I
\end{split}
\end{equation}

Let ${\bf C}={\bf V}{\bf X}$, we transform this problem to the following regression problem,

\begin{equation}
\begin{split}
&\argmin_{{\bf X},{\bf C}}\ ||{\bf V}-{\bf C}{\bf X}^T||_F+\mu||{\bf C}||_{1,1}\\
&s.t.\ {\bf X}^T{\bf X}={\bf I}
\end{split}
\label{equation:step2}
\end{equation}

Eqn. \ref{equation:step2} can be solved by using the alternating minimization scheme as in \cite{SPCA-2006}, with initial ${\bf X}$ being an identity matrix. Initially, we set $\mu = ||{\bf V}||_F/||{\bf V}{\bf X}||_{1,1}$ to emphasize more on the sparsity objective, and gradually degrade $\mu$ to a small value to ensure ${\bf C}$ spanning the same subspace as ${\bf V}$ through the last iterations.

Adding the global sparsity optimization step to ASPCA-F and ASPCA-B, we have two new models ASPCA-FG and ASPCA-BG, respectively. Algorithms 1 and 2 summarize the entire optimization process, where ${\bf A}$ is the covariance matrix of the input dataset, $d$ is the number of sparse principal components extracted from the abnormal subspace, $max\_iter$ is the number of iterations for the global sparsity optimization, and the output loading matrix {\bf V} contains $d$ orthogonal and sparse loading vectors for detecting and interpreting anomalies.

\renewcommand{\algorithmicrequire}{\textbf{Input:}}
\renewcommand{\algorithmicensure}{\textbf{Output:}}
\begin{algorithm}
\begin{algorithmic}[1]
    \caption{Forward Abnormal Subspace Sparse PCA with Global Optimization (ASPCA-FG)}
    \REQUIRE  ${\bf A},\ d,\ \lambda$, and $max\_iter$
    \ENSURE ${\bf V}$
    \FOR{$i=1$ to $p$}
	\STATE ${\bf V}_{i-1} \gets ({\bf v}_1,{\bf v}_2...{\bf v}_{i-1})$;
 	\STATE ${\bf R}_{i-1} \gets {\bf V}_{i-1}{\bf V}_{i-1}^T$;
	\STATE Optimize ${\bf v}_i$ with given ${\bf A},\ {\bf R}_{i-1}$ according to Eqn.~\ref{equation:aspca_f_sdp_final} using SDP;
    \ENDFOR
    \STATE ${\bf V} \gets ({\bf v}_{p-d+1},...{\bf v}_p)$;
    \STATE Optimize ${\bf X},\ {\bf C}$ with given ${\bf V},\ max\_iter$ according to Eqn.~\ref{equation:step2} using the alternating minimization scheme;
    \STATE ${\bf V} \gets {\bf C}{\bf X}$;
    \RETURN ${\bf V}$;
\end{algorithmic}
\end{algorithm}

\renewcommand{\algorithmicrequire}{\textbf{Input:}}
\renewcommand{\algorithmicensure}{\textbf{Output:}}
\begin{algorithm}
\begin{algorithmic}[1]
    \caption{Backward Abnormal Subspace Sparse PCA with Global Optimization (ASPCA-BG)}
    \REQUIRE  ${\bf A},\ d,\ \lambda$, and $max\_iter$
    \ENSURE ${\bf V}$
    \FOR{$i=1$ to $d$}
	\STATE ${\bf V}_{i-1} \gets ({\bf v}_1,{\bf v}_2...{\bf v}_{i-1})$;
 	\STATE ${\bf R}_{i-1} \gets {\bf V}_{i-1}{\bf V}_{i-1}^T$;
	\STATE Optimize ${\bf v}_i$ with given ${\bf A},\ {\bf R}_{i-1}$ according to Eqn.~\ref{equation:aspca_b_sdp_final} using SDP;
    \ENDFOR
    \STATE ${\bf V} \gets ({\bf v}_1,{\bf v}_2...{\bf v}_d)$;
    \STATE Optimize ${\bf X},\ {\bf C}$ with given ${\bf V},\ max\_iter$ according to Eqn.~\ref{equation:step2} using the alternating minimization scheme;
    \STATE ${\bf V} \gets {\bf C}{\bf X}$;
    \RETURN ${\bf V}$;
\end{algorithmic}
\end{algorithm}


\newpage
\section{Experiment}

\subsection{Datasets}
Our proposed ASPCA models were evaluated on the synthetic data introduced in Section~2, a medical dataset {\bf Breast-Cancer}, and a network intrusion detection dataset {\bf KDD99}.

{\bf Breast Cancer Wisconsin (Diagnostic) Data Set} \cite{breast-cancer-data} provides features to distinguish malignant and benign tumors. The features describe characteristics of the cell nuclei present in a digitized image of a \emph{fine needle aspirate} (FNA) of a breast mass. As there are plenty cells for a breast mass, the features are three important statistics (mean, standard error, and worst value) on 10 features for each cell: \emph{radius}, \emph{texture}, \emph{perimeter}, \emph{area}, \emph{smoothness}, \emph{compactness}, \emph{concavity}, \emph{concave points}, \emph{symmetry} and \emph{fractal dimension}. There are 357 benign records, and we kept the first 10 malignant records to transform the classification task to an anomaly detection task, same as the other works \cite{Kriegel:2009:LLO, Amer:2013:EOS} did.  All 30 real-valued features were deducted by the mean values and linearly scaled to $[-1, 1]$.

{\bf KDD 99 Intrusion Dataset} \cite{cup1999data} is a widely used data for anomaly and intrusion detection. Each instance is a connection record classified as normal or one of 22 classes of attacks. Attacks fall into four main groups: DoS, Remote-to-local, User-to-root, and Probe. We chose all normal and part of the abnormal records in 10\% KDD99 datasets as shown in Table \ref{Table:KDD 99} as picking the first 500 records on \emph{smurf} ,\emph{neptune}, \emph{back}, \emph{teardrop}, \emph{satan}, \emph{ipsweep}, and \emph{portsweep} and all records on other attacking types. The number of records for each type are shown in brackets in Table \ref{Table:KDD 99} too.We followed a similar preprocessing procedure as in \cite{jiang2013family}. There are 41 features including seven categorical features which were mapped into distinct positive integers from 0 to $m-1$ ($m$ is the number of states for the categorical feature). For example, 0 to 2 in \emph{protocol\_type} stands for TCP, UDP, and ICMP. Logarithmic scaling was applied on \emph{duration},  \emph{src\_bytes}, and \emph{dst\_bytes}, and all features were deducted by the mean values and linearly scaled to $[-1, 1]$. 

\begin{table}
\small
\centering
\caption{Statistics of KDD99 and the relevant features identified by JSPCA}
\label{Table:KDD 99}
\begin{tabular}{|c|c|c|c|}
\hline
Categories & \begin{tabular}[c]{@{}c@{}}\# \\Records \end{tabular}  & Types                                                                                                                                                             & \begin{tabular}[c]{@{}c@{}}Relevant Features \\ (JSPCA)\end{tabular}                                                                           \\ \hline
Normal     & 97,277 &                                                                                                                                                                   &                                                                                                                                                \\ \hline
DoS        & 2,264  & \begin{tabular}[c]{@{}c@{}}smurf(500),\\neptune(500),\\ back(500),\\ teardrop(500),\\ pod(264)\end{tabular}                                                      & \begin{tabular}[c]{@{}c@{}}\emph{service}, \\ \emph{src\_bytes},\\ \emph{dst\_bytes},\\ \emph{count}, \emph{srv\_count},\\ \emph{dst\_host\_count},\\ \emph{dst\_host\_srv\_count}\end{tabular} \\ \hline
Probe      & 1,731  & \begin{tabular}[c]{@{}c@{}}satan(500),\\ ipsweep(500),\\ portsweep(500),\\ nmap(231)\end{tabular}                                                                 & \emph{source\_bytes}                                                                                                                                  \\ \hline
U2R        & 52     & \begin{tabular}[c]{@{}c@{}}buffer\_overflow(30),\\ loadmodule(9),\\ perl(3), rootkit(10)\end{tabular}                                                             & \begin{tabular}[c]{@{}c@{}}\emph{duration},\\ \emph{src\_bytes},\\ \emph{dst\_bytes},\\ \emph{dst\_host\_count},\\ \emph{dst\_host\_srv\_count}\end{tabular}                      \\ \hline
R2L        & 1,126  & \begin{tabular}[c]{@{}c@{}}ftp\_write(8),\\ guess\_passwd(53),\\ imap(12), \\ multihop(7),\\ phf(4), spy(2),\\ warezclient(1,020),\\ warezserver(20)\end{tabular} & \begin{tabular}[c]{@{}c@{}}\emph{duration}, \emph{service},\\ \emph{src\_bytes},\\ \emph{dst\_bytes},\\ \emph{dst\_host\_count},\\ \emph{dst\_host\_srv\_count}\end{tabular}             \\ \hline
\end{tabular}
\end{table}

\subsection{Methodology}

We compared our proposed ASPCA models with the standard PCA model for detection and sparsity performance,  and with two state-of-the-art analytical models on PCA results: the JSPCA model \cite{jiang2013family} and a decision tree model used in \cite{XuWei-SOSP} for interpretation performance. For the decision tree model, we formed the training set with all predicted normal records and anomalies returned by our ASPCA model as negative and positive samples, respectively. Then the decision trees were trained using the CART model from MATLAB and trimmed manually for the best interpretation.

The parameters used in our model were listed in Table~\ref{table:parameters} and discussed in Section~\ref{sec:parameter}. The number of PCs used in PCA equals the total number of features minus the number of abnormal PCs for all datasets. The results of JSPCA on the synthetic data were obtained by choosing the best performed parameters, and we directly reported their results on KDD99 in \cite{jiang2013family}.   Note that, our ASPCA models make no assumptions on the anomalies in the dataset, and we built one model on the entire KDD99 dataset with anomalies from all different categories, whereas JSPCA built four models on KDD99, each including anomalies for a single major attacking category \cite{jiang2013family}. 

\begin{table}
\small
\centering
\caption{Parameters}
\begin{tabular}{|c|c|c|}
\hline
              & \# abnormal PCs & $\lambda$ \\ \hline
Synthetic     & 4               & 5         \\ \hline
Breast-Cancer & 10              & 5         \\ \hline
KDD99        & 35              & 100       \\ \hline
\end{tabular}
    \label{table:parameters}
\end{table}  

Finally, we implemented all methods with MATLAB and CVX, and performed all experiments on a laptop computer with 16 GB memory and a Intel(R) Core(TM) i7-4870HQ 2.50GHz CPU.

\subsection{Experimental Results}

\subsubsection{Detection Evaluation}  
We first compared the various ASPCA models with the standard PCA model on the anomaly detection performance. Because all models can obtain a perfect ROC
curve for the Synthetic data, we only show the ROC curves on Breast-Cancer and KDD99 in Figure~\ref{fig:cancer:roc} and Figure~\ref{fig:kdd99:roc}, respectively. Note that, 
since ASPCA-F and ASPCA-FG use the same abnormal subspace to detect anomalies, their ROC curves are identical which are labeled as ASPCA-F(G). Similarly, the ROC curves of ASPCA-B and ASPCA-BG are labeled as ASPCA-B(G). From Figure~\ref{fig:cancer:roc}, we can see that our proposed ASPCA-F(G) and ASPCA-B(G) models performed similarly or even better than PCA on anomaly detection for both datasets.

\begin{figure}
\centering
\subfloat[Breast-Cancer]
{
	\label{fig:cancer:roc}
	\includegraphics[width=40mm]{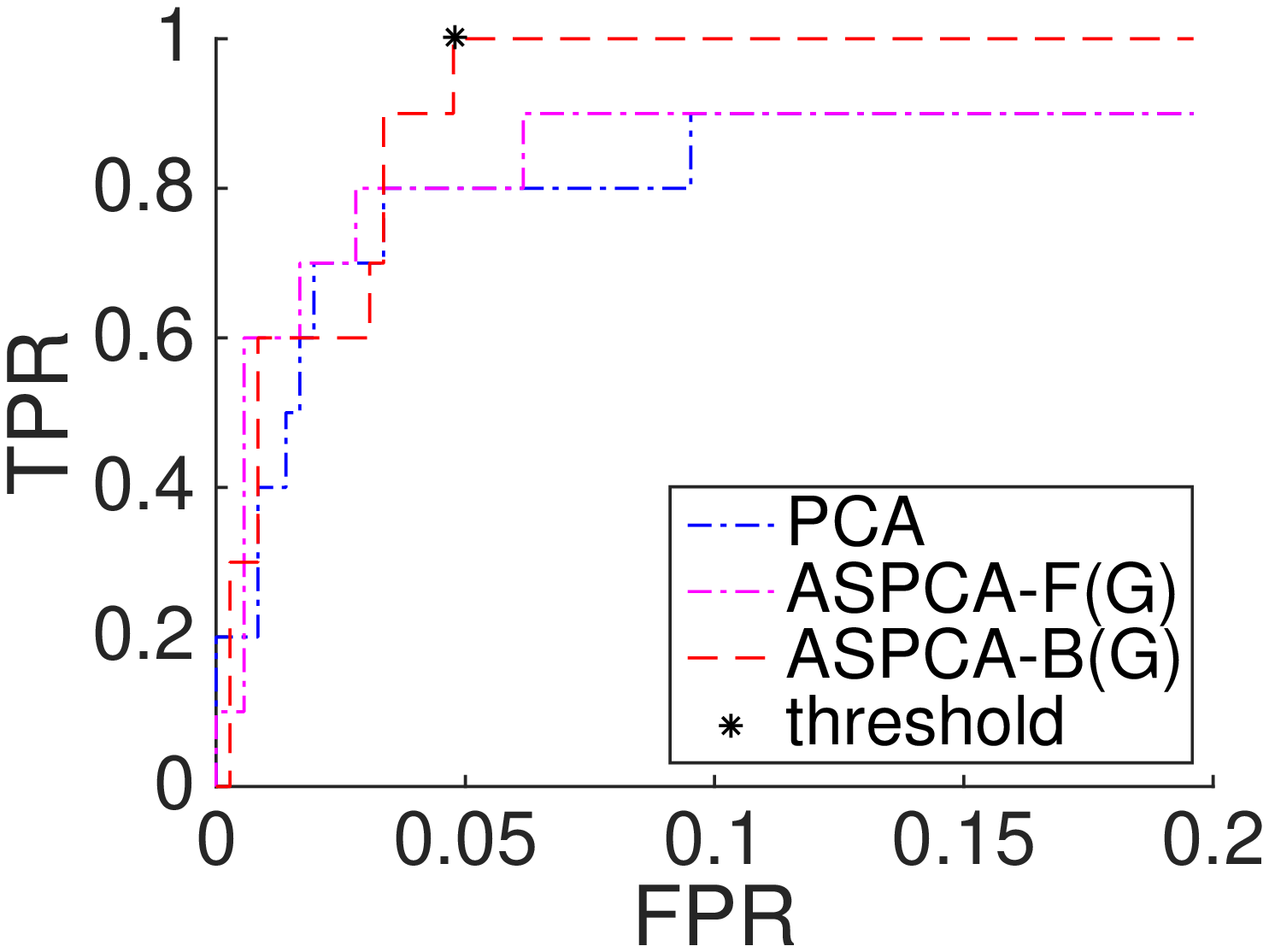}
}
\subfloat[KDD99]
{
	\label{fig:kdd99:roc}
	\includegraphics[width=40mm]{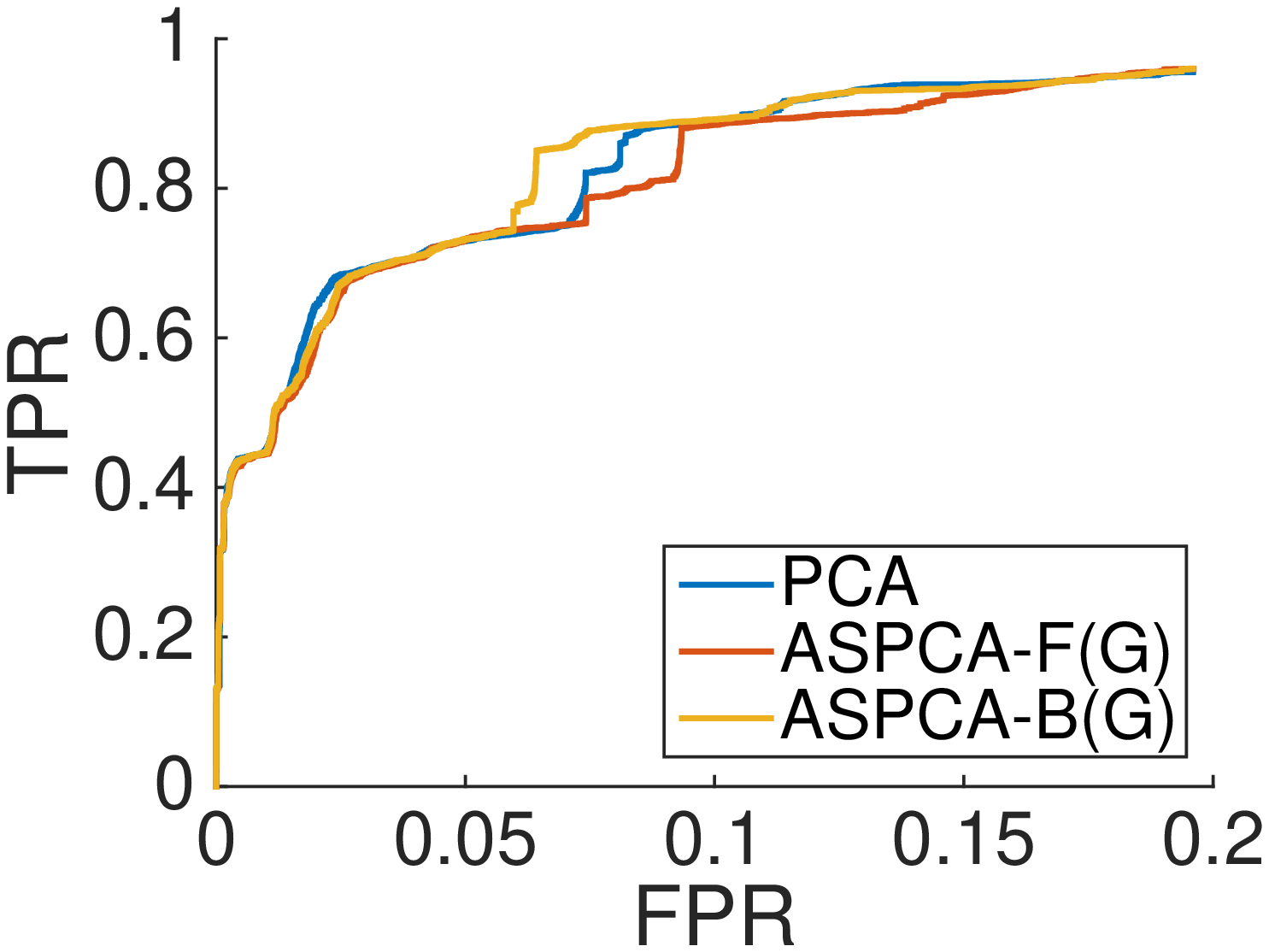}
}
\caption{ROC curves}
\label{fig:detection performane}
\end{figure}

\subsubsection{Sparsity Evaluation}
The next set of experiments were designed to compare the sparsity of the loading matrix generated by various ASPCA models and we used the result of PCA as our baseline.  We used three metrics to evaluate the sparsity of the loading matrix of the abnormal PCs, namely,  $||V||_{1,1}$, $Card_{0.1}$ (number of entries with absolute values bigger than 0.1), and $Card_{0.01}$ (number of entries with absolute values bigger than 0.01), and showed the results in Table~\ref{table:sparsity}. We can see that all ASPCA models improved the sparsity of the loading matrix greatly over the baseline. For all datasets, the ASPCA-B model achieved better sparsity performance than the ASPCA-F model. The global optimization step improved $||V||_{1,1}$ values for both models on Breast-Cancer and KDD99. However, in terms of cardinality, ASPCA-FG performed worse than ASPCA-F on Breast-Cancer. The global optimization step achieved the largest sparsity improvement on KDD99, as it has more abnormal PCs than the other two datasets leaving more room for the global optimization. Overall, the ASPCA-BG model achieved the best sparsity performance.
  
The loading matrices returned by PCA and ASPCA-B (the other three ASPCA models have very similar results) on the Synthetic data are shown in Figure~\ref{fig:mappingmatrix}, and the loading matrices returned by PCA, ASPCA-F, ASPCA-B, ASPCA-FG, ASPCA-BG on Breast-Cancer and KDD99 are shown in Figure~\ref{fig:components}. We can see that ASPCA-F usually leaves some loading vectors with poor sparsity towards the end, which should be avoided as they are part of abnormal subspace. On the contrary, ASPCA-B leaves the loading vectors not so sparse towards the beginning, which need no interpretation as they belong to the normal subspace. 

\begin{table}
\small
\caption{Sparsity on Synthetic data, Breast-Cancer, and KDD99}
\begin{tabular}{|l|l|l|l|l|}
\hline
dataset                        & method   & $||V||_{1,1}$ & $card_{0.1}$ & $card_{0.01}$ \\ \hline
\multirow{5}{*}{Synthetic}     & PCA      & 7.07    & 16        & 24           \\ \cline{2-5} 
		                       & ASPCA-F  & 5.54    & 9           & 9           \\ \cline{2-5} 
		                       & ASPCA-B  & 5.31    & 8          & 8           \\ \cline{2-5} 
		                       & ASPCA-FG & 5.54    & 9           & 9           \\ \cline{2-5} 
		                       & ASPCA-BG & 5.31    & 8           & 8           \\ \hline
\multirow{5}{*}{Breast-Cancer} & PCA      & 34.22    & 111        & 237           \\ \cline{2-5} 
		                       & ASPCA-F  & 17.23    & 33           & 50           \\ \cline{2-5} 
		                       & ASPCA-B  & 12.31    & 18          & 18           \\ \cline{2-5} 
		                       & ASPCA-FG & 16.50    & 36           & 63           \\ \cline{2-5} 
		                       & ASPCA-BG & 12.31    & 18           & 18           \\ \hline
\multirow{5}{*}{KDD99}        & PCA      & 97.33    & 248          & 691           \\ \cline{2-5} 
                               & ASPCA-F  & 55.01    & 96           & 265           \\ \cline{2-5} 
                               & ASPCA-B  & 54.52    & 101          & 215           \\ \cline{2-5} 
                               & ASPCA-FG & 42.77    & 58           & 159           \\ \cline{2-5} 
                               & ASPCA-BG & 43.05    & 57           & 157           \\ \hline
\end{tabular}
\label{table:sparsity}
\end{table}

\begin{figure*}
	\centering
    \subfloat[]{\includegraphics[width=35mm]{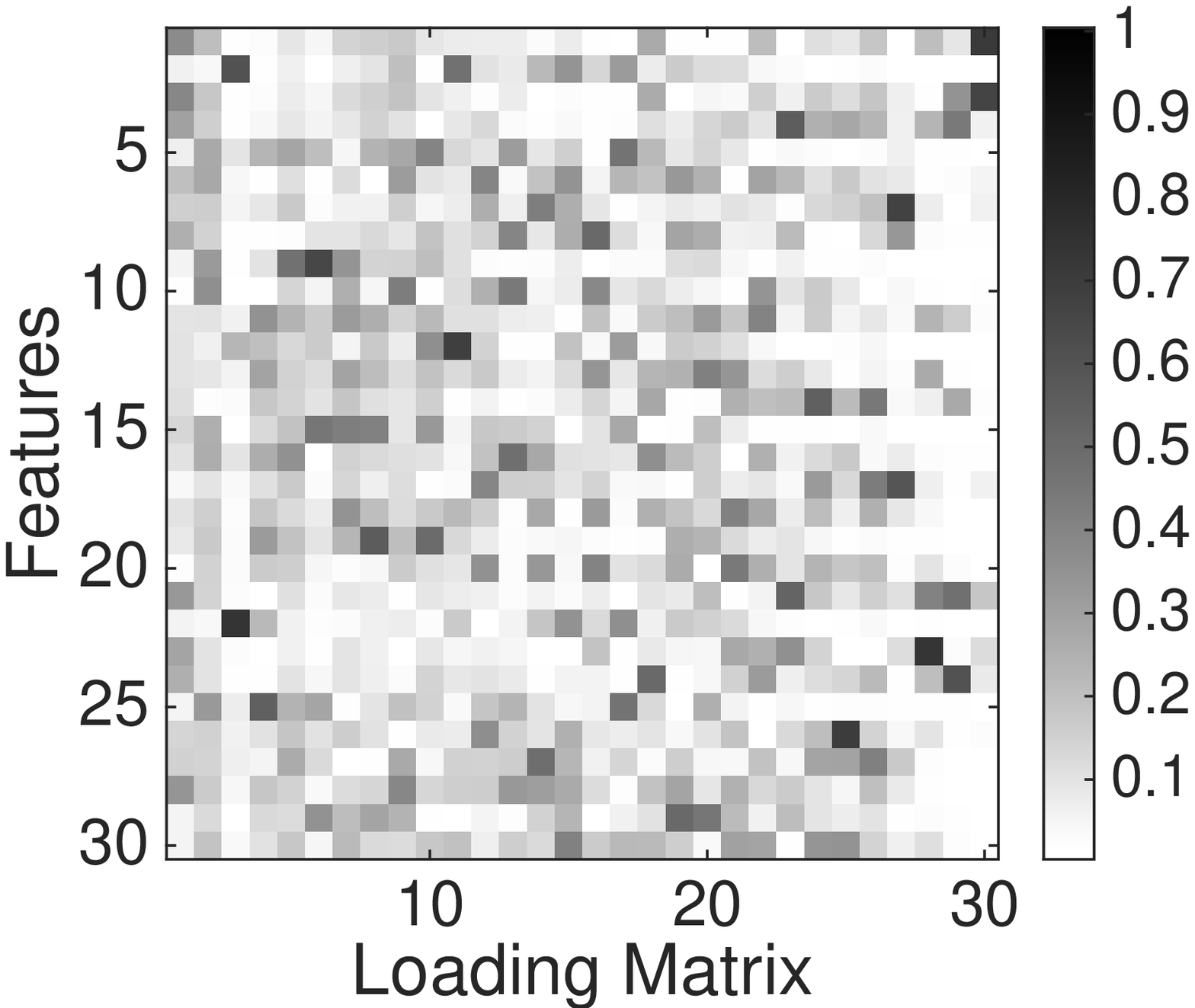}}
    \subfloat[]{\includegraphics[width=35mm]{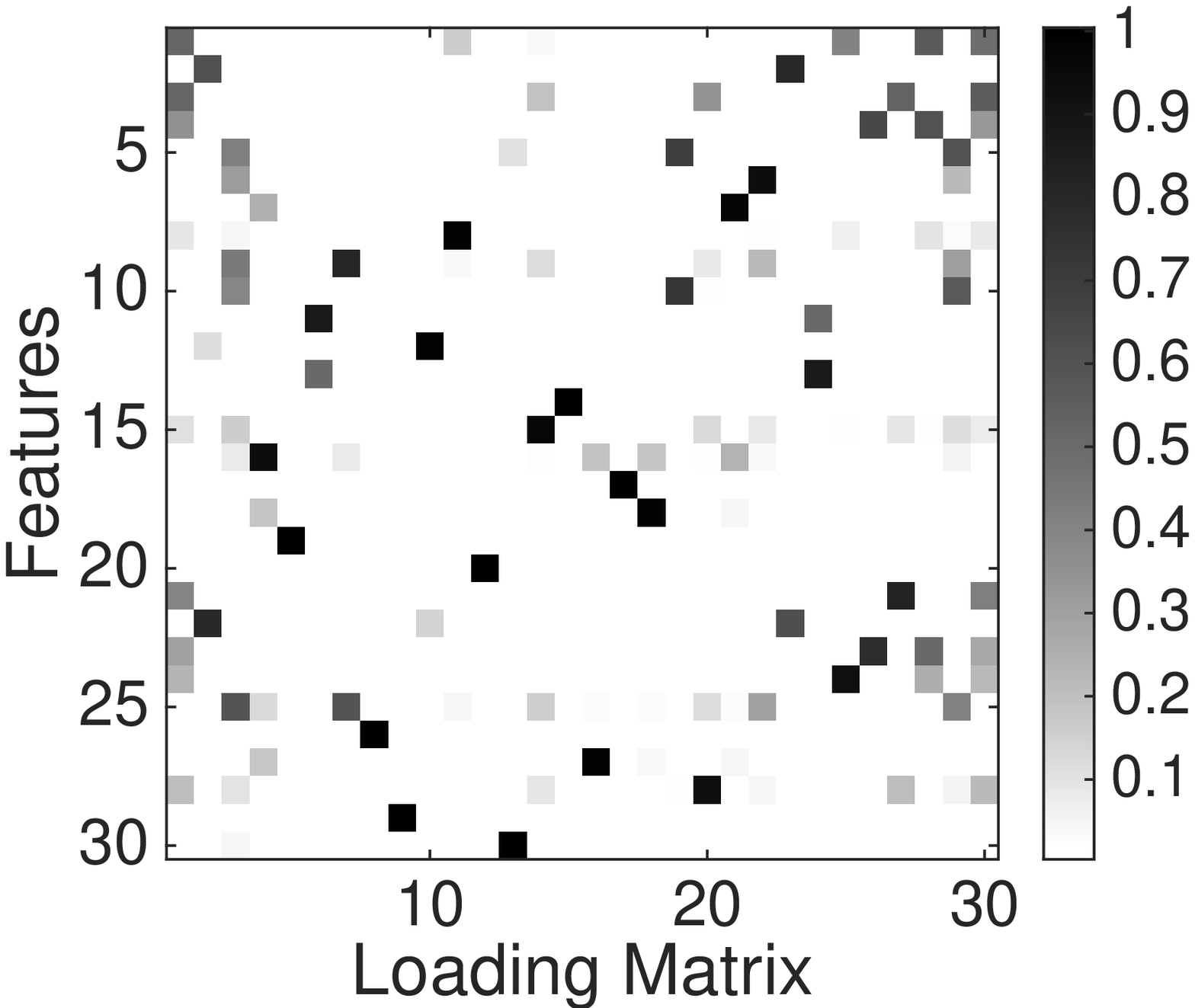}}
    \subfloat[]{\includegraphics[width=35mm]{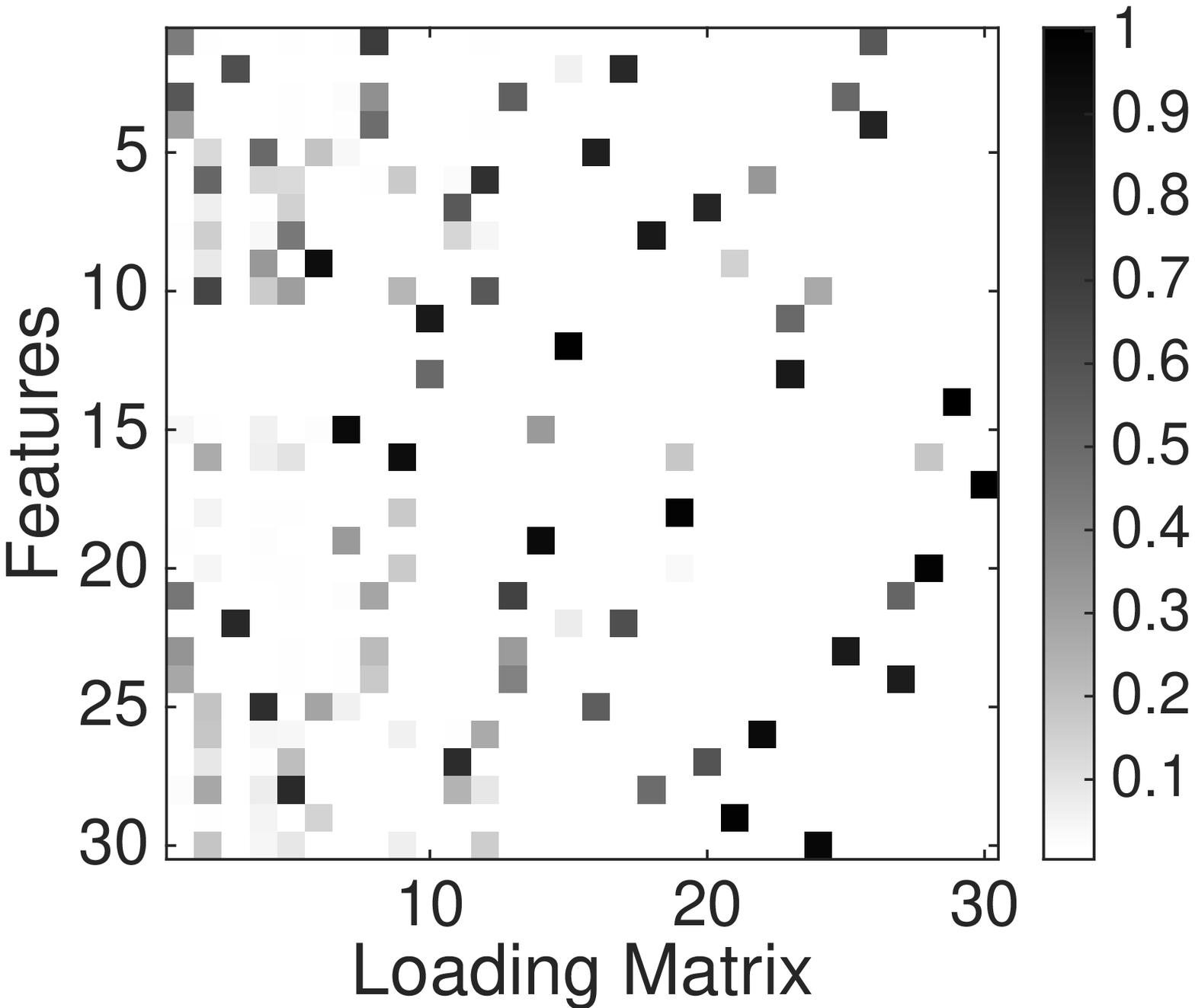}}
    \subfloat[]{\includegraphics[width=35mm]{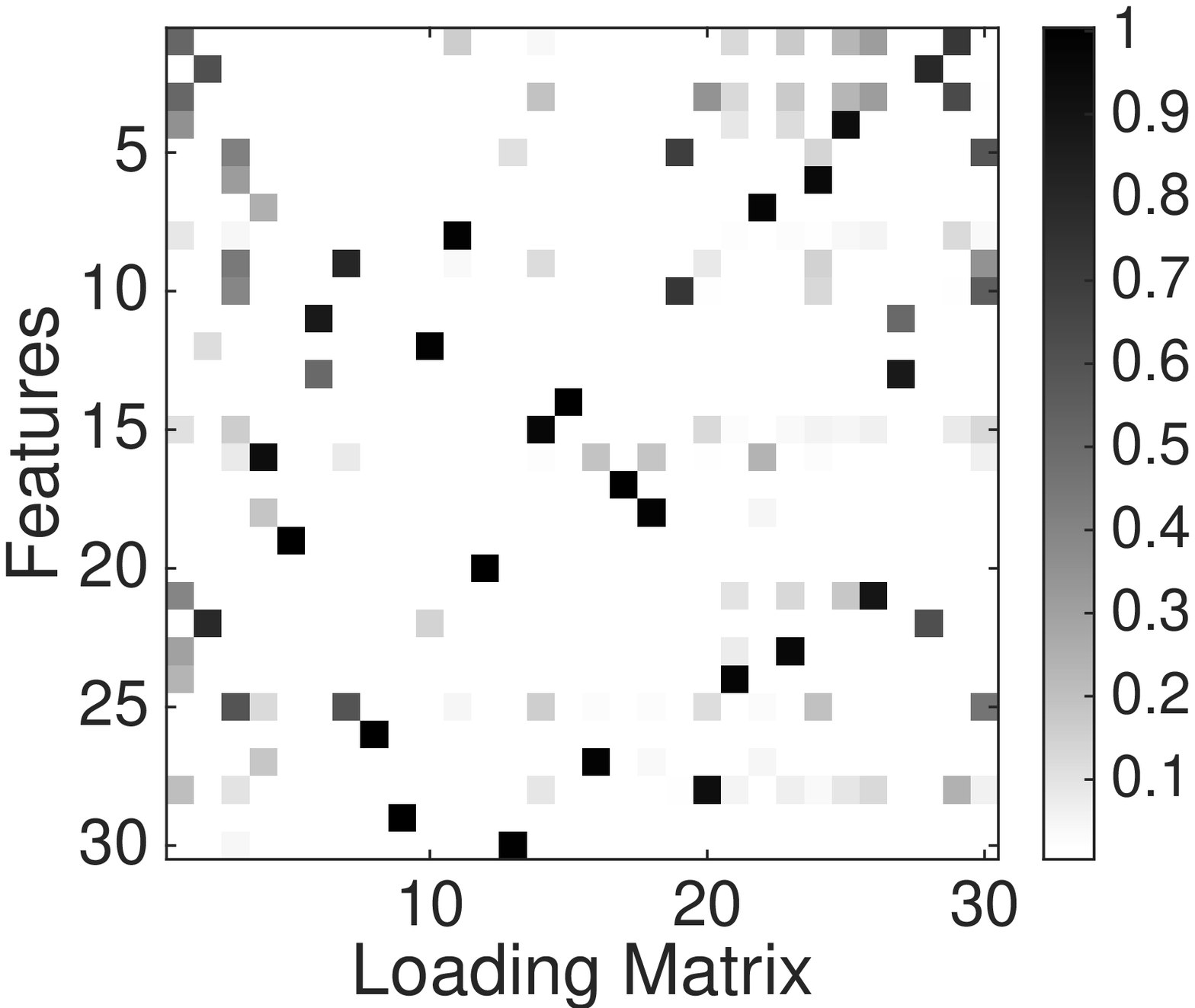}}
    \subfloat[]{\includegraphics[width=35mm]{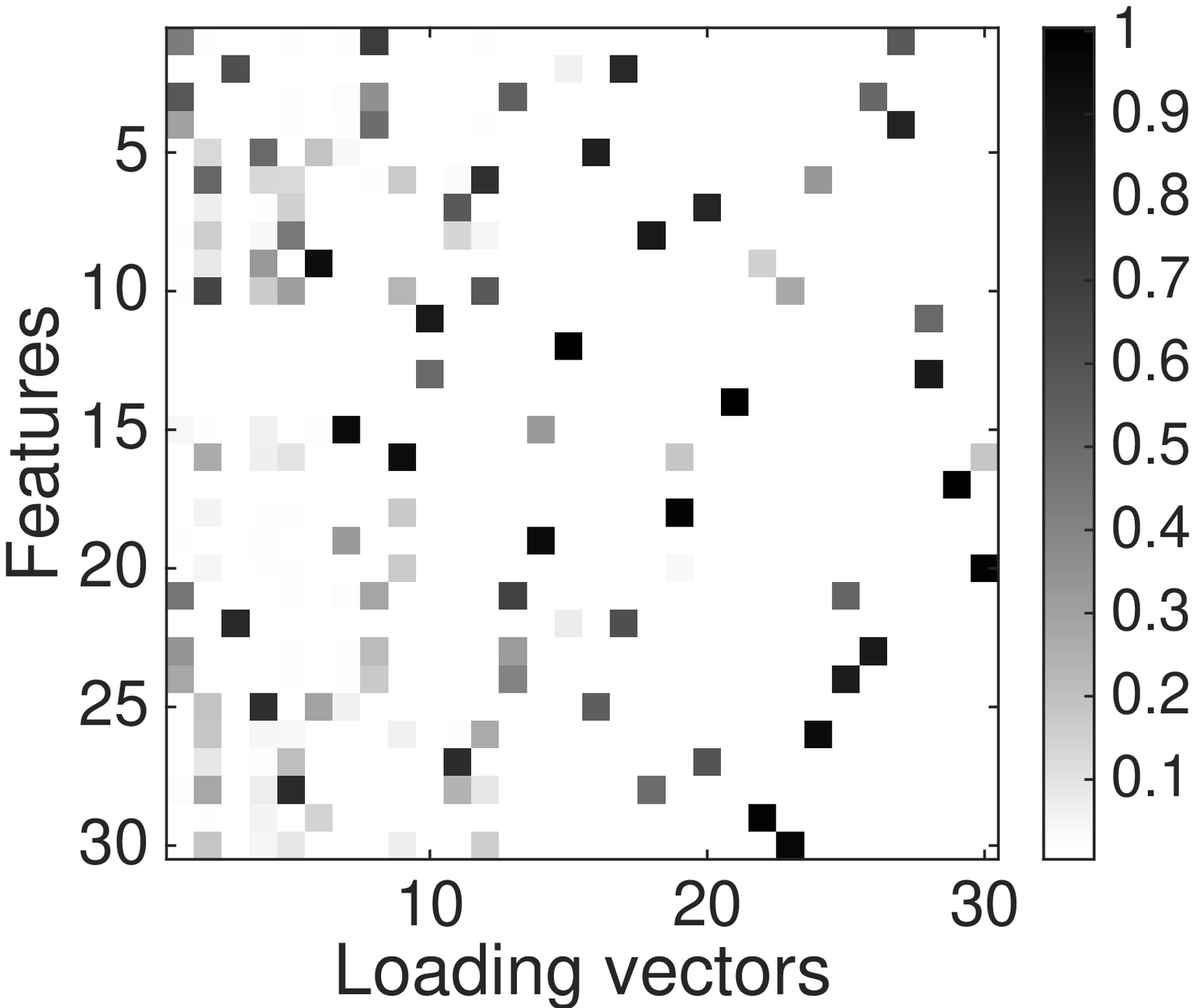}}
	\\
    \subfloat[]{\includegraphics[width=35mm]{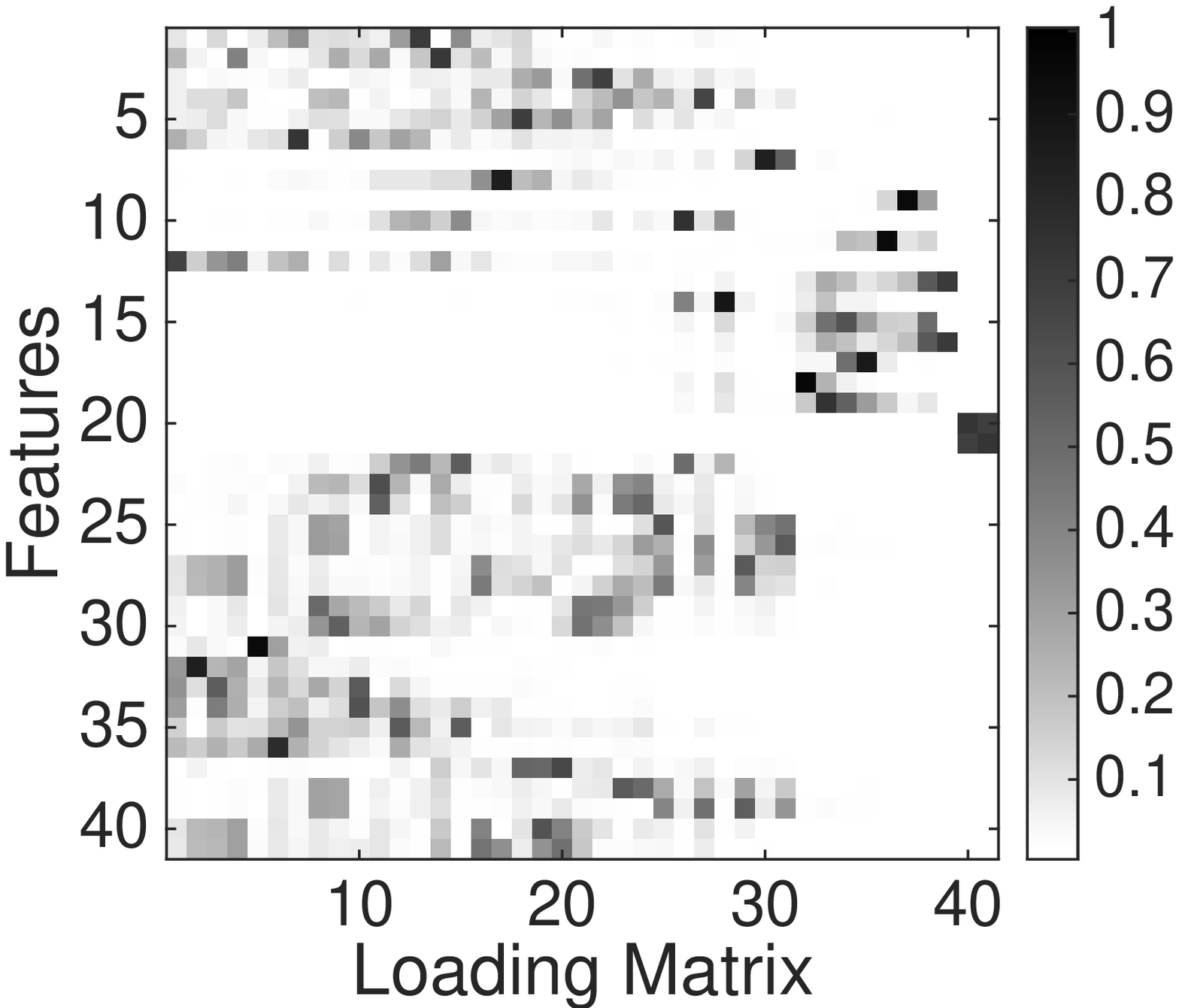}}
    \subfloat[]{\includegraphics[width=35mm]{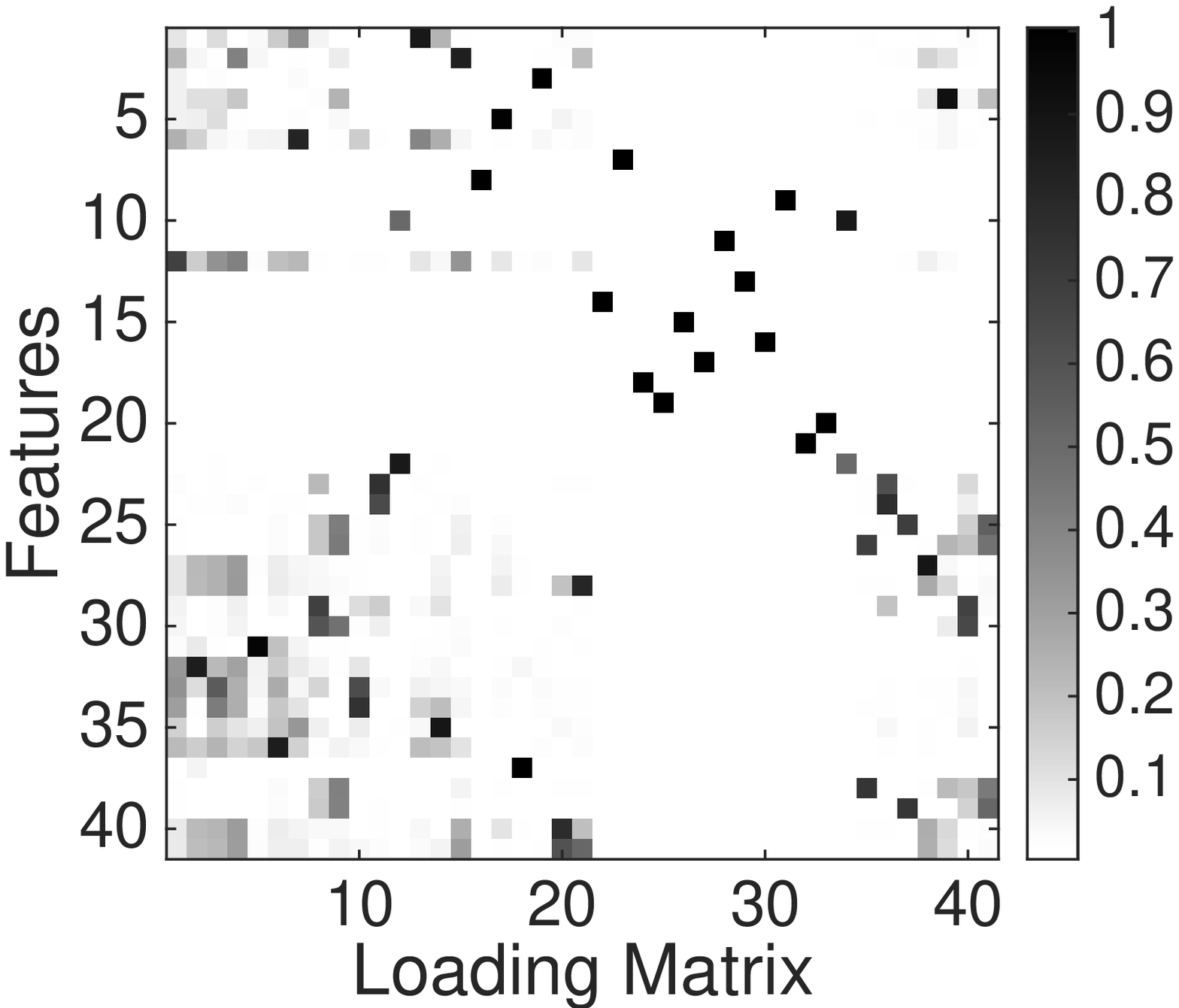}}
    \subfloat[]{\includegraphics[width=35mm]{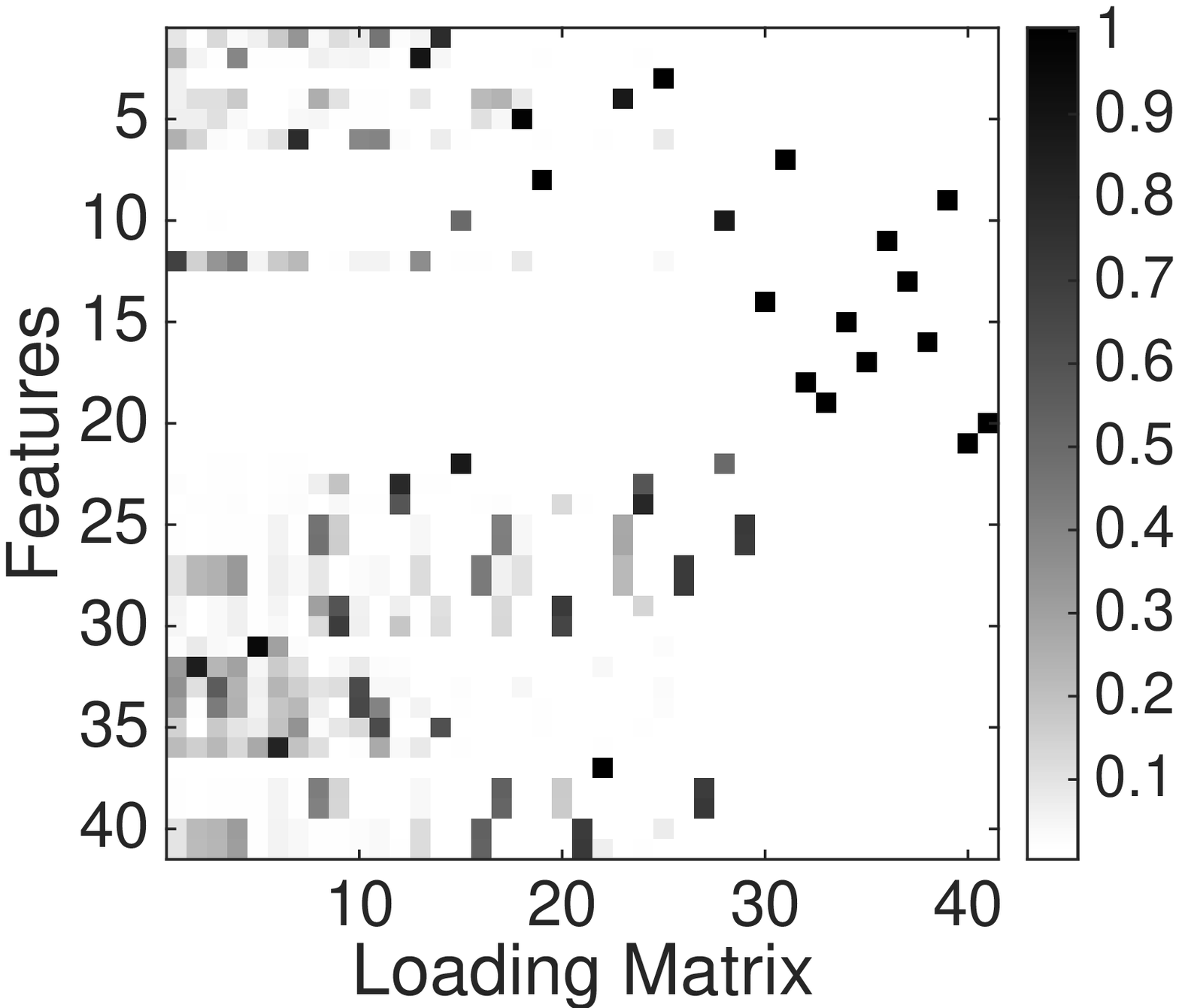}}
    \subfloat[]{\includegraphics[width=35mm]{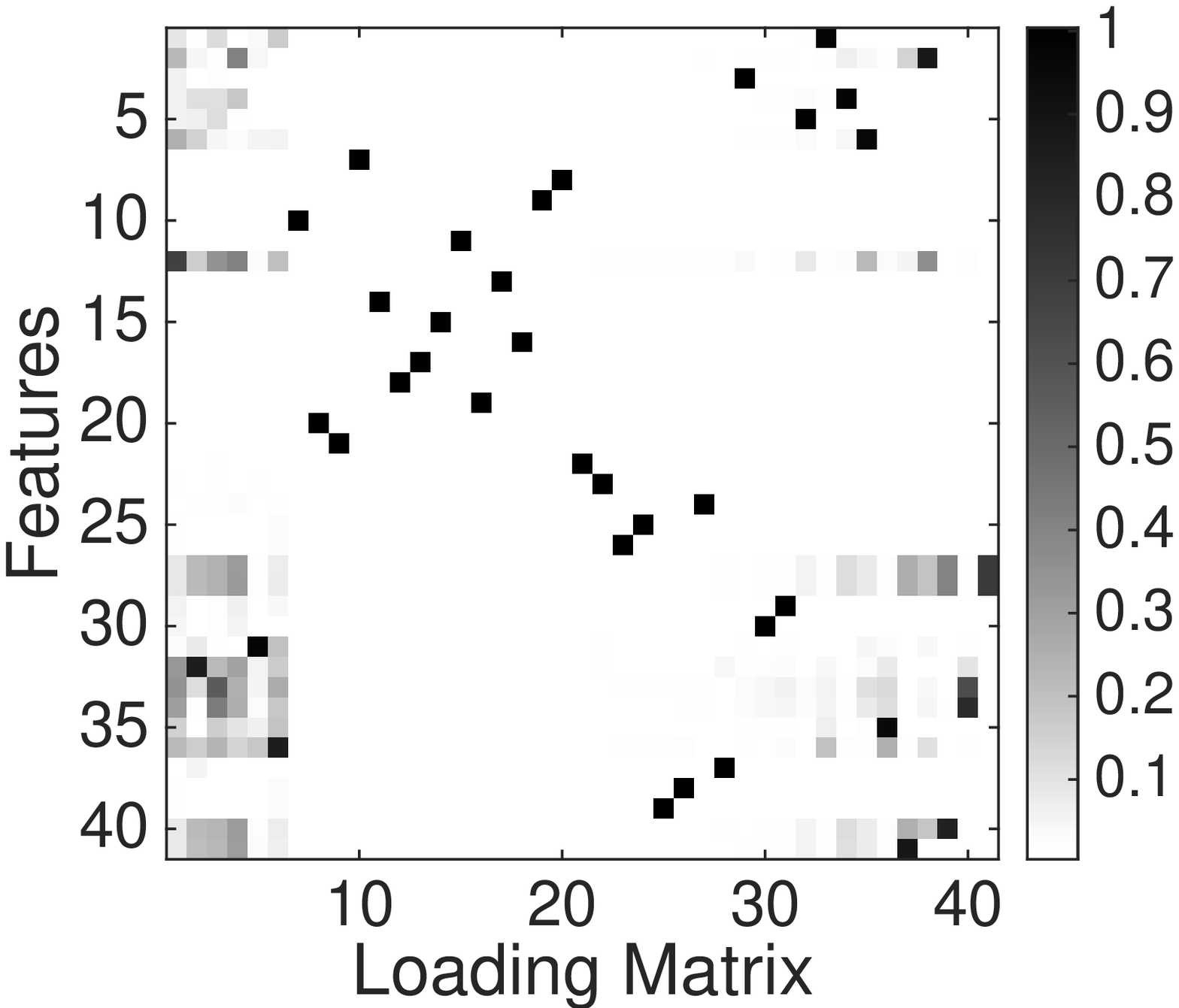}}
    \subfloat[]{\includegraphics[width=35mm]{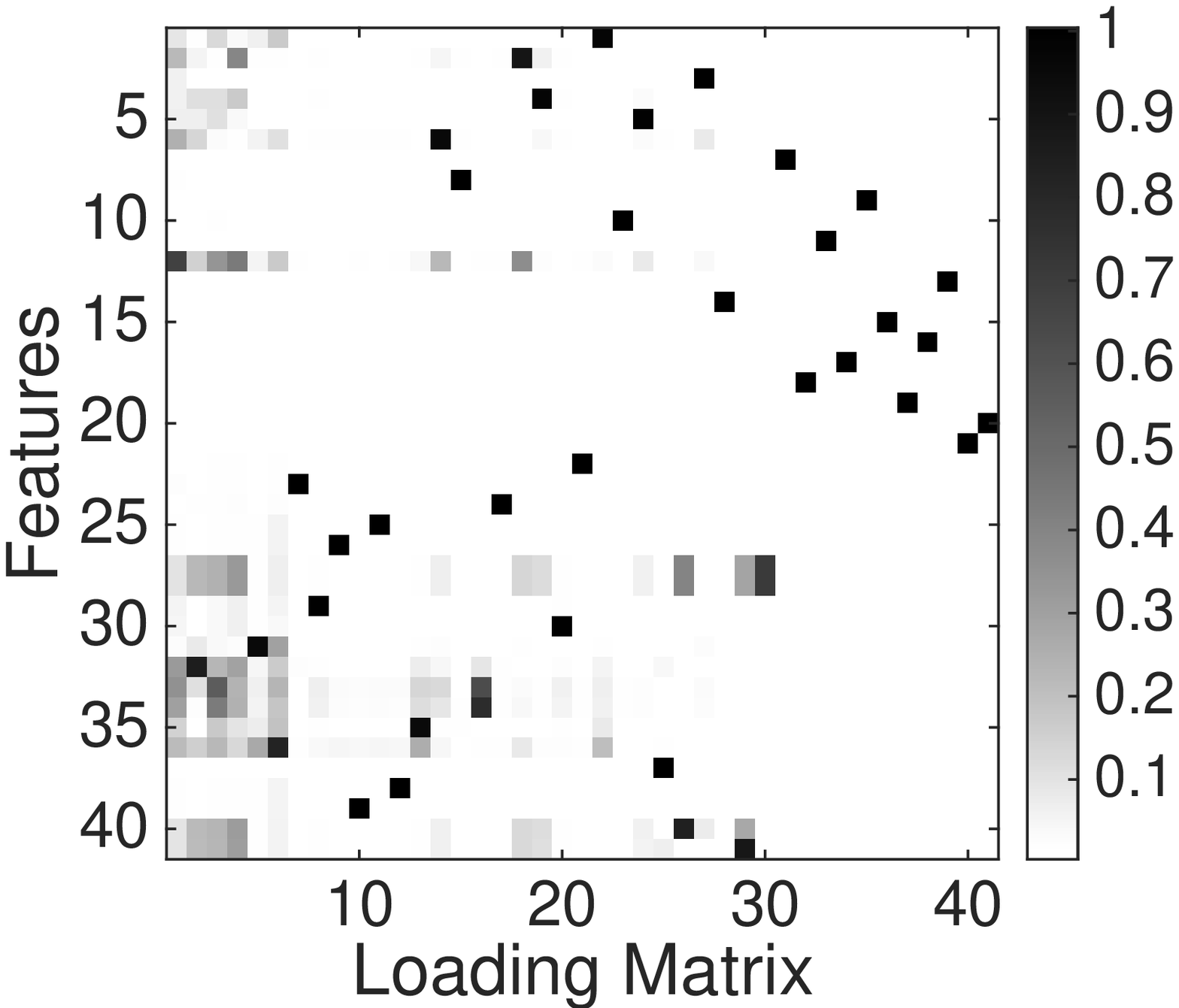}}
\caption{Loading matrix of PCs of Breast-Cancer (Top) and KDD99 (Bottom) obtained by PCA, ASPCA-F, ASPCA-B, ASPCA-FG, and ASPCA-BG from left to right}
\label{fig:components}
\end{figure*}

\subsubsection{Interpretation Evaluation}

Now we evaluate the interpretation performance of the ASPCA-BG model, as it has the best sparsity performance. Since we want to see how true anomalies are interpreted by our model, we selected a threshold value on SPE to ensure most of the true anomalies are detected. We show the threshold values, false positive rates (FPR), and true positive rates (TPR) for all three datasets in Table~\ref{table:threshold}. 

\begin{table}
\small
\centering
\caption{SPE Threshold}
\begin{tabular}{|l|l|l|l|}
\hline
              & SPE threshold & TPR    & FPR    \\ \hline
Synthetic     & 0.25          & 1      & 0      \\ \hline
Breast-Cancer & 0.1003  & 1      & 0.0476 \\ \hline
KDD99        & 0.5075      & 0.8516 & 0.0657 \\ \hline
\end{tabular}
\label{table:threshold}
\end{table}

{\bf Synthetic Data:}
The four abnormal PCs and the projection values of 15 anomalies on these PCs are shown in Table~\ref{table:components:synthetic} and Figure~\ref{fig:heatmap:synthetic}, respectively. As shown in Table~\ref{table:components:synthetic}, the first three PCs correspond to the rules of $D \approx C + A$, $A \approx B$,  and $F \approx 0$, respectively. The anomalies breaking these rules indeed have large projection values on the corresponding PCs. Thus, our ASPCA-BG model can not only identify the set of
features that are responsible for an anomaly, but also tell the cause of the anomaly, \i.e., breaking the rules indicated by the abnormal PCs.  

JSPCA also successfully identified the relevant features $(A, B, D, F)$ as suggested in Figure~\ref{fig:loading matrix:synthetic}. However, it cannot tell the source of each individual anomaly. Unlike JSPCA, our ASPCA models make no assumptions on whether there are anomalies present in the dataset for model training. Keeping only normal data from the Synthetic dataset,  the ASPCA-BG model found four abnormal PCs with loading vectors $(0.31, 0.31, 0.64, 0.64, 0, 0, 0)^T$, $(-0.71,$  $0.71, 0, 0, 0, 0)^T$, $(0, 0, 0, 0, 0, 1, 0)^T$ , and $(0, 0, 0, 0, 0, 0, 1)^T$, which are very similar to the ones in Table~\ref{table:components:synthetic}. Using these abnormal PCs, we can successfully detect and interpret anomalies as well.

\begin{figure}
\centering
	\includegraphics[width=42mm]{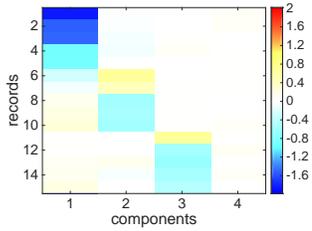}
\caption{Heatmap of projection values of anomalies on abnormal PCs for Synthetic Data}
\label{fig:heatmap:synthetic}
\end{figure}

\begin{table}
\centering
\caption{Components on Synthetic Data}
\label{table:components:synthetic}
\small
\begin{tabular}{|c|c|}
\hline
Index & Components \\ \hline
1 & 0.3099 A + 0.3122 B + 0.6370 C - 0.6330 D \\
2 & 0.7095 A - 0.7047 B\\
3 & 1 F\\
4 & 1 G\\
\hline
\end{tabular}
\end{table}

\begin{figure}
	\centering
	\includegraphics[width=60mm]{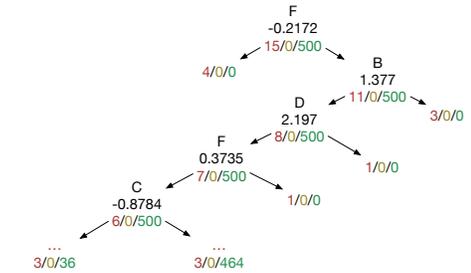}
	\caption{Decision tree on Synthetic Data}
\label{fig:dt:synthetic}
\end{figure}

The decision tree was shown in Figure~\ref{fig:dt:synthetic}, where on each node we showed the feature and its value used to partition the data, the number of true positives detected by ASPCA-BG in red, the number of false positives in yellow, and the number of normal data detected by ASPCA-BG in green. We can see that the decision tree model needs several rules to describe a group of anomalies which could be easily described by a clear linear combination and a threshold. In Figure~\ref{fig:dt:synthetic}, only the third type of anomalies, which has a large absolute value on F, is easy for the decision tree model to interpret.

{\bf Breast-Cancer:}
The projection values of 10 anomalies on the abnormal PCs obtained by ASPCA-BG for Breast-Cancer are shown in Figure~\ref{fig:heatmap:cancer}.  We can see that the malignant records have two patterns: the first four records have large projection values on the 2nd, 3rd, and 4th PCs, the rest records have large projection values on the first PC and moderate projection values on the 6th PC. We show these PCs in Table~\ref{table:breast_cancer}, where coefficients in PCs less than 0.1 were omitted.  The features appearing in the 1st and 6th PCs are \emph{area\_se}, \emph{area\_worst}, and \emph{radius\_worst}, which were reported previously by \cite{Wolberg1995792} as being effective for classifying malignant records.  Note that, \emph{area\_worst} actually has a quadratic relation with \emph{radius\_worst}, our model identified it as a linear relation, which is a good approximation in a small range of radius. The first four records, on the other hand, do not have large projection values on PCs related to \emph{area} features.  They were detected by PCs related to \emph{symmetry}, \emph{fractal dimension}, and \emph{compactness} features, which clearly indicates another type of malignant records.

\begin{figure}
\centering
	\includegraphics[width=50mm]{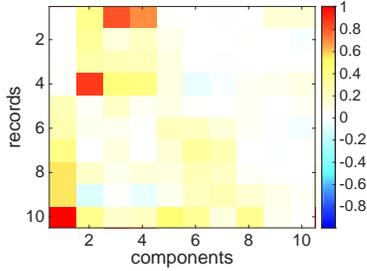}
\caption{Heatmap of projection values of anomalies on abnormal PCs of Breast-Cancer}
\label{fig:heatmap:cancer}
\end{figure}

\begin{table}
\caption{Components on Breast-Cancer.}
\label{table:breast_cancer}
\small
\centering
\begin{tabular}{|c|c|}
\hline
Index & Components                                                                                                  \\ \hline
1     & 1 \emph{area\_se}                                                                                       \\ \hline
2     & 0.9894 \emph{symmetry\_worst}             \\ \hline
3     & \begin{tabular}[c]{@{}c@{}}0.9631 \emph{fractal\_dimension\_worst} \\ - 0.2693 \emph{fractal\_dimension\_mean}\end{tabular}                                                                                                 \\ \hline
4     & \begin{tabular}[c]{@{}c@{}}0.9445 \emph{compactness\_worst}\\  - 0.3286 \emph{compactness\_mean}\end{tabular}                                                                                              \\ \hline
6     & 0.8554 \emph{area\_worst} - 0.5180 \emph{radius\_worst}                                                                                \\ \hline
\end{tabular}
\end{table}

\begin{figure}
\centering
	\includegraphics[width=42mm]{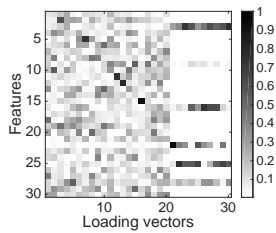}
\caption{Loading matrix of abnormal PCs obtained by JSPCA on Breast-Cancer}
\label{fig:JSPCA:cancer}
\end{figure}

\begin{figure}
\centering
	\includegraphics[width=40mm]{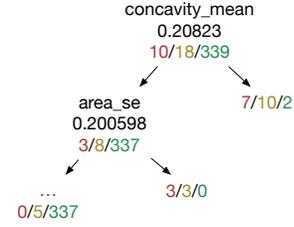}
	\caption{Decision tree on Breast-Cancer}
         \label{fig:tree:cancer}
\end{figure}

The loading matrix of the abnormal PCs obtained by JSPCA on Breast-Cancer is shown in Figure~\ref{fig:JSPCA:cancer}. The relevent features are \emph{radius\_mean}, \emph{concavity\_mean}, \emph{area\_se}, \emph{fractal\_} \emph{dimension\_se}, \emph{perimeter\_worst}, and \emph{compactness\_worst}. As we can see, JSPCA cannot tell the different causes of individual anomalies.
The decision tree obtained on Breast-Cancer is shown in Figure~\ref{fig:tree:cancer}. Our ASPCA-BG model detected 10 true positives, 18 false positives, and 339 true negatives (shown on the root node in red, yellow, and green, respectively) with the chosen SPE threshold. The tree used \emph{concavity\_mean} to separate positive and negative samples. However, the attribute \emph{concavity\_mean} is orthogonal to the abnormal subspace obtained by our ASPCA-BG model, which means it is not the feature based on which our ASPCA-BG model detects anomalies. Hence, using decision trees to interpret results of a subspace-based model, such as ours, may lead to misleading interpretations.

{\bf KDD99:} 
With the given SPE threshold, our ASPCA-BG model detected 4397 true positives on KDD99. Although our model is intended to analyze individual anomalies, we can also summarize interpretations of similar anomalies to make our discussion easier. We used a simple way to generate signatures on whether an anomaly has \emph{low} ($\leq -\sqrt{SPE threshold/2}$), or \emph{high} ($\geq \sqrt{SPE threshold/2}$) projection values on the set of abnormal PCs. Then anomalies were grouped according to their signatures, so that the components in the signature of each group is common for most of the anomalies in the group. 

In Table~\ref{table:signatures on KDD}, we listed some major signatures found by the above method. Actually, for most of the cases, we can associate each signature group with an anomaly type quite well.  The two numbers listed for each anomaly type are the number of detected anomalies by our model and the number of total anomalies of this type, respectively. The two numbers listed for each signature group are the number of the anomalies of this type and the total number of anomalies in the group, respectively. Some anomaly types only have one main corresponding signature groups, whereas we identified three main signature groups for warezclient{[}R2L{]}. For the $i$th PC, $i$L and $i$H represent low and high projection values on it, respectively.  Finally, the components appeared in these signatures are shown in Table~\ref{table:components on KDD}.

\begin{table}
\small
\caption{Signatures on KDD99}
\label{table:signatures on KDD}
\centering
\begin{tabular}{|c|c|l|}
\hline
                                                                                         & \# in Type/              & Important       \\
Type                                                                                     & \# in Group              & Components      \\ \hline
\multirow{2}{*}{\begin{tabular}[c]{@{}c@{}}neptune{[}DoS{]}\\ 500/500\end{tabular}}      & \multirow{2}{*}{481/484} & 2L, 9H, 10H     \\
                                                                                         &                          & 11H, 12H, 14H   \\ \hline
\begin{tabular}[c]{@{}c@{}}smurf{[}DoS{]}\\ 493/500\end{tabular}                         & 472/472                  & 1H, 3L, 8H, 16H \\ \hline
\begin{tabular}[c]{@{}c@{}}teardrop{[}DoS{]}\\ 496/500\end{tabular}                      & 393/393                  & 18H             \\ \hline
\begin{tabular}[c]{@{}c@{}}satan{[}Probe{]}\\ 500/500\end{tabular}                       & 444/444                  & 2L, 3H, 6H, 8H  \\ \hline
\begin{tabular}[c]{@{}c@{}}portsweep{[}Probe{]}\\ 354/500\end{tabular}                  & 175/175                  & 6H, 7L          \\ \hline
\begin{tabular}[c]{@{}c@{}}ipsweep{[}Probe{]}\\ 217/500\end{tabular}                     & 104/109                  & 1H, 19H         \\ \hline
\multirow{3}{*}{\begin{tabular}[c]{@{}c@{}}warezclient{[}R2L{]}\\ 974/1020\end{tabular}} & 272/395                  & 15H, 20H        \\ \cline{2-3} 
                                                                                         & 332/339                  & 4L, 5H, 6L      \\ \cline{2-3} 
                                                                                         & 237/292                  & 4L, 5H          \\ \hline
\begin{tabular}[c]{@{}c@{}}guess-passwd{[}R2L{]}\\ 53/53\end{tabular}                    & 48/121                   & 4H, 5H          \\ \hline
\begin{tabular}[c]{@{}c@{}}buffer-overflow{[}U2R{]}\\ 30/30\end{tabular}                 & 7/7                      & 5H, 24H         \\ \hline
\end{tabular}
\end{table}

From Table~\ref{table:signatures on KDD}, we can see that the signatures for different anomaly types varied a lot, from which we often can find the components that are consistent with the nature of each anomaly type. For example, smurf[DoS] attacks are also known as popular form of DoS packet floods, which turn out to have high \emph{srv\_count} and \emph{count} values ({\it i.e.}, 16H and 8H). Teardrop[DoS] attacks try to break the host by sending mangled IP fragments, which led to high \emph{wrong\_fragment} values ({\it i.e.}, 18H).  We can see similar trends for Probe attacks too. For example, ipsweep[Probe] attacks sweep different hosts (IPs) to find cracks for hacking ({\it i.e.}, 19H), whereas portsweep[Probe] attacks try to visit different service ({\it i.e.}, 6H) and short connection duration ({\it i.e.}, 7L). Buffer overflow[U2R] attackers try to gain the root authority on the host, and 24H indicates that the user has logged into the server with root shell. Warezclient[R2L] attackers try to download files in forbidden directories from the FTP servers. Our interpretation is consistent with \cite{sabhnani2003kdd} as \emph{logged\_in} with low \emph{dst\_bytes} ({\it i.e.}, 4L), \emph{is\_guest\_login} ({\it i.e.}, 15H) and high \emph{hot} values ({\it i.e.}, 20H). 

\begin{table}
\small
\caption{Components on KDD99}
\label{table:components on KDD}
\centering
	\begin{tabular}{|c|c|}
\hline
Index & Components                                                                                                                   \\ \hline
1     & 0.8906 \emph{protocol\_type} + 0.3658 \emph{logged\_in}                                                                                    \\ \hline
2     & 0.9949 \emph{same\_srv\_rate}                                                                                                       \\ \hline
3     & 0.9958 \emph{diff\_srv\_rate}                                                                                                       \\ \hline
4     & \begin{tabular}[c]{@{}c@{}}0.9520 \emph{dst\_bytes} \\ - 0.2222 \emph{logged\_in}\end{tabular}                                             \\ \hline
5     & \begin{tabular}[c]{@{}c@{}}0.7722 \emph{dst\_host\_same\_srv\_rate} \\ - 0.6286 \emph{dst\_host\_srv\_count}\end{tabular}                  \\ \hline
6     & \begin{tabular}[c]{@{}c@{}}0.9466 \emph{dst\_host\_diff\_srv\_rate}\end{tabular}       \\ \hline
7     & 0.9714 \emph{duration}                                                                                                              \\ \hline
8     & 0.9995 \emph{count}                                                                                                                 \\ \hline
9     & 0.9716 \emph{flag}                                                                                                                  \\ \hline
10    & 0.9984 \emph{dst\_host\_serror\_rate}                                                                                               \\ \hline
11    & 0.9981 \emph{srv\_serror\_rate}                                                                                                     \\ \hline
12    & 0.9981 \emph{serror\_rate}                                                                                                          \\ \hline
14    & 0.9985 \emph{dst\_host\_srv\_serror\_rate}                                                                                          \\ \hline
15    & 0.9997 \emph{is\_guest\_login}                                                                                                      \\ \hline
16    & 0.9994 \emph{srv\_count}                                                                                                            \\ \hline
18    & 0.9996 \emph{wrong\_fragment}                                                                                                       \\ \hline
19    & 0.9970 \emph{dst\_host\_srv\_diff\_host\_rate}                                                                                      \\ \hline
20    & 0.9999 \emph{hot}                                                                                                                   \\ \hline
24    & 1 \emph{root\_shell}                                                                                                                \\ \hline
	\end{tabular}
\end{table}

The results obtained by JSPCA are shown in Table~\ref{Table:KDD 99}. The selected features by JSPCA are more general and similar for all categories. Some of the important features for specific attacks are missing too, for instance, \emph{root\_shell} for User-2-Root[U2R] attacks and \emph{hot} and \emph{is\_guest\_login} for R2L attacks as mentioned in \cite{sabhnani2003kdd}.

We show the decision tree obtained on KDD99 in Figure~\ref{fig:tree:KDD}. The features captured by the decision tree are consistent with the components discovered by our ASPCA-BG model to a large extent. For example, low \emph{same\_srv\_rate} was chosen to detect neptune and satan DoS attackers, which is the same as using 2L in our model to detect the same attacks. Similarly, high \emph{protocol\_type} values ({\it i.e.}, using ICMP protocol) was chosen to detect ipsweep and smurf attackers (detected by 1H in our model). Low \emph{dst\_host\_srv\_count} with high \emph{hot} values is an important character of warezclient and guess-passed attacks  (identical to 5H and 20H in our model). An interesting observation on the decision tree in Figure~\ref{fig:tree:KDD} is that a node on the tree will stop splitting as soon as the samples in the node are mostly positive or negative ones. For example, the left child node of the root in Figure~\ref{fig:tree:KDD} stopped splitting with anomalies from different attack types, in which case, our model can provide more information on the differences of the these attack types.
\begin{figure}
\centering
	\includegraphics[width=70mm]{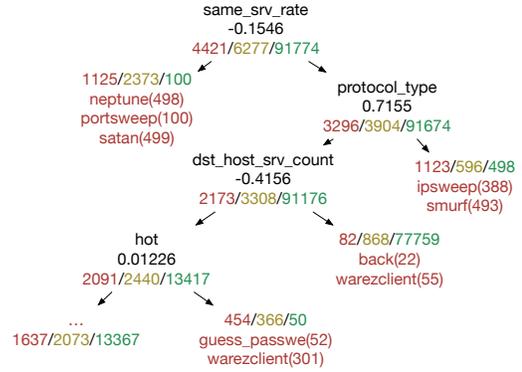}
	\caption{Decision tree on KDD99}
         \label{fig:tree:KDD}
\end{figure}

\subsection{Parameter Selection}
\label{sec:parameter}
Our ASPCA models has two parameters to select: the number of abnormal PCs and the coefficient $\lambda$ on sparsity. We know that PCA-based anomaly detection methods are sensitive to the number of PCs \cite{PCA-Sensitivity}. We plotted the detection accuracy (in terms of Area Under ROC Curve (AUC)) with different number of abnormal PCs on Breast-Cancer (with $\lambda=5$) and KDD99 ($\lambda=100$) in Figure~\ref{Figure:parameter:components}. We selected 20 normal PCs ({\it i.e.}, 10 abnormal PCs) for Breast-Cancer data and 6 normal PCs ({\it i.e.}, 35 abnormal PCs), for KDD99 data, to achieve the highest AUC values for our baseline method PCA, to make the comparisons on detection accuracy fair.

The coefficient $\lambda$ is a trade-off between the sparsity and the additional variance on components. With less additional variance, the variances on the whole detection space for various ASPCA models are closer to the one of PCA. We showed the variance, sparsity (valued by $||V||_{1,1}$) and AUC values obtained by varying the value of $\lambda$ on KDD99 and Breast-Cancer in Table~\ref{Table:parameter:lambada} for ASPCA-FG and ASPCA-BG. We can see that our models are not very sensitive to $\lambda$  in terms of AUC, and we selected $\lambda=5$ for Breast-Cancer, $\lambda=100$ for KDD99 for moderate sparsity and variance. The trends are similar for ASPCA-F and ASPCA-B, which were omitted due to space constraints. We selected the same $\lambda$ values for ASPCA-F and ASPCA-B, as in ASPCA-FG and ASPCA-BG, respectively.

\begin{figure}
\centering
\subfloat[Breast-Cancer]
{
	\includegraphics[width=40mm]{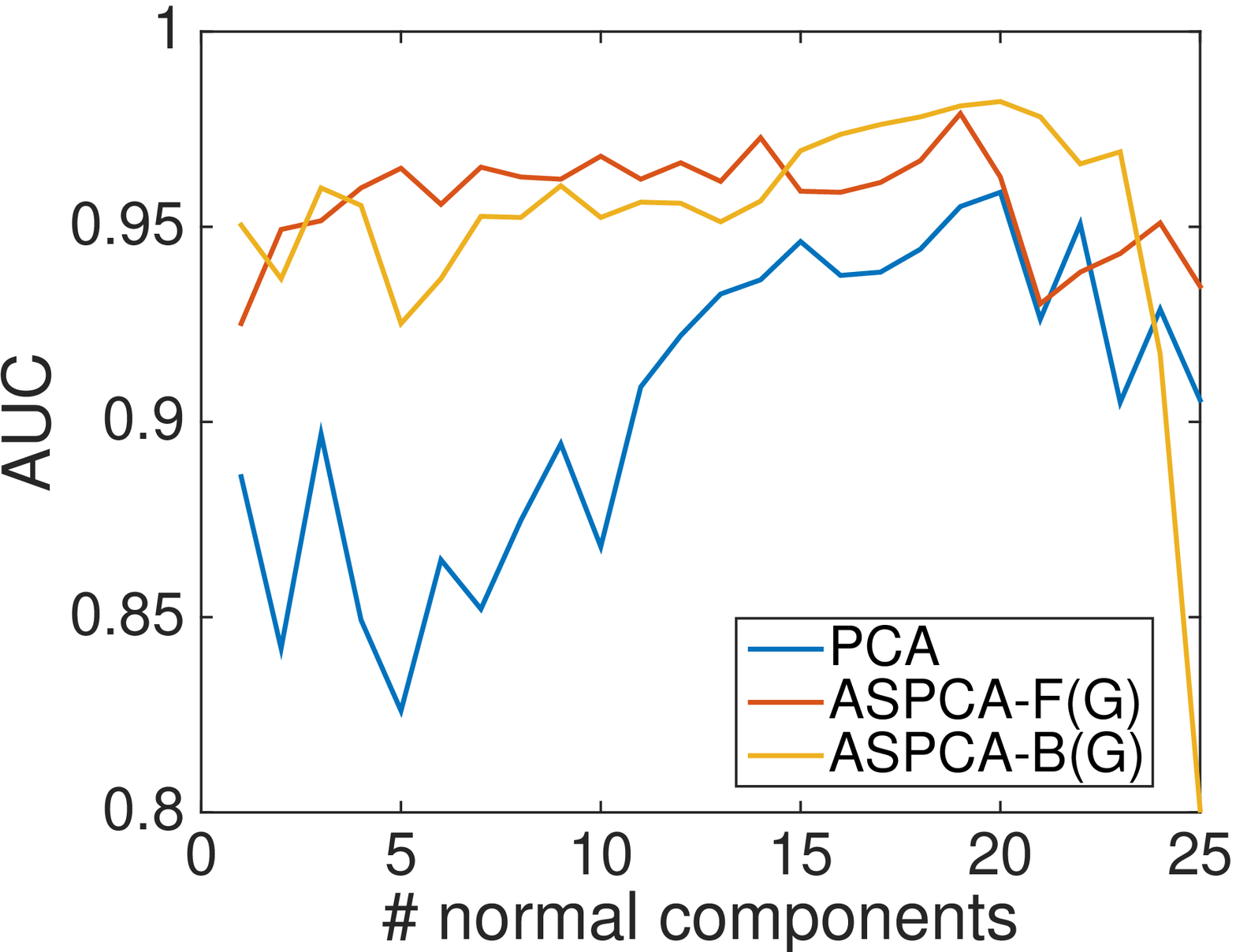}
}
\subfloat[KDD99]
{
	\includegraphics[width=40mm]{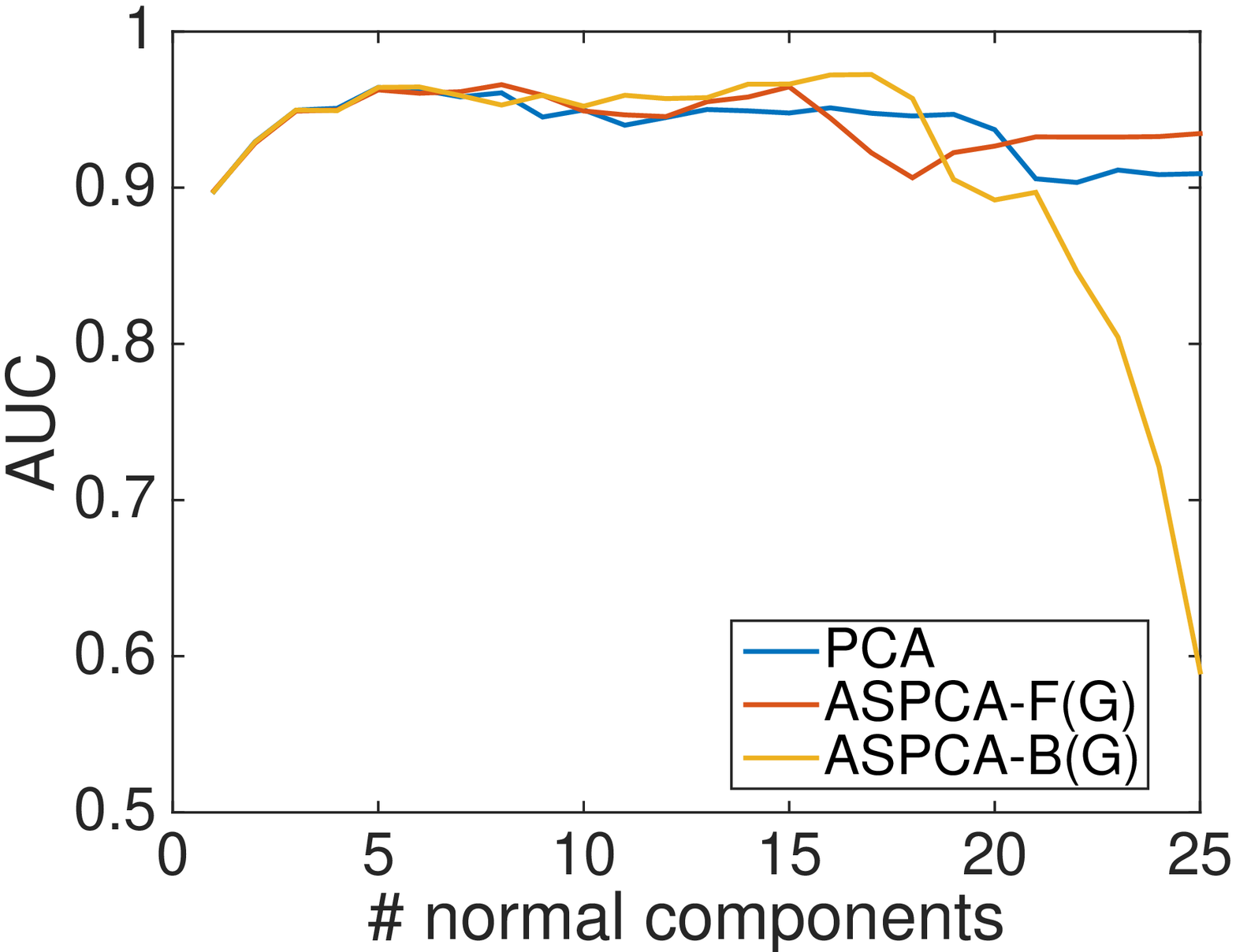}
}
\caption{Selection on the number of normal PCs}
\label{Figure:parameter:components}
\end{figure}

\begin{table*}
\small
\centering
\caption{Selection on $\lambda$}
\subfloat[KDD99]{
\begin{tabular}{|l|l|l|l|l|l|l|}
\hline
\multicolumn{1}{|l|}{} & \multicolumn{3}{c|}{ASPCA-FG}    & \multicolumn{3}{c|}{ASPCA-BG}    \\ \hline
$\lambda$ & $||V||_{1,1}$ & Variance & AUC   & $||V||_{1,1}$ & Variance & AUC   \\ \hline
0         & 97.33         & 21518    & 0.963 & 97.33         & 21518    & 0.963 \\
10        & 44.38         & 21524    & 0.963 & 44.59         & 21523    & 0.963 \\
50        & 43.52         & 21627    & 0.962 & 43.97         & 21589    & 0.964 \\
100       & 42.77         & 21873    & 0.960 & 43.04         & 21735    & 0.964 \\
500       & 41.04         & 23673    & 0.956 & 40.67         & 22997    & 0.967 \\ \hline
\end{tabular}
}
\subfloat[Breast-Cancer]{
\begin{tabular}{|l|l|l|l|l|l|l|}
\hline
\multicolumn{1}{|l|}{} & \multicolumn{3}{c|}{ASPCA-FG}    & \multicolumn{3}{c|}{ASPCA-BG}    \\ \hline
$\lambda$              & $||V||_{1,1}$ & Variance & AUC   & $||V||_{1,1}$ & Variance & AUC   \\ \hline
0                      & 34.23         & 1.2728   & 0.959 & 34.23         & 1.2728   & 0.959 \\
1                      & 26.68         & 14.2012  & 0.903 & 14.95         & 6.2777   & 0.950 \\
5                      & 16.50         & 20.8308  & 0.963 & 12.31         & 20.2968  & 0.982 \\
10                     & 12.81         & 30.8123  & 0.985 & 10            & 57.0009  & 0.966 \\
50                     & 10            & 57.0009  & 0.966 & 10            & 57.0009  & 0.966 \\ \hline
\end{tabular}
}
	\label{Table:parameter:lambada}

\end{table*}

\section{Conclusions and Future Work}

Traditional PCA-based anomaly detection models are not suitable for anomaly interpretation, limiting its usage in the domains where interpretation is essential.  In this paper, we found that the sparsity and orthogonality of the loading vectors are the keys to anomaly interpretation, and proposed an interpretable PCA-based anomaly detection model, the ASPCA model. We designed \emph{forward} and \emph{backward} ASPCA models and evaluated them on two real world datasets. Our model achieved similar or even better anomaly detection performance as the traditional PCA model, and provided meaningful interpretation for individual anomalies. Our future works will focus on three directions: 1) how to improve efficiency on high dimensional datasets; 2) how to extend our model to robust PCA for better detection performance; 3) how to extend our model to kernel PCA.

\section{Acknowledgement}
This work is funded by the National 863 Program of China (Grant No. 2015AA01A301).

\bibliographystyle{abbrv}
\small
\bibliography{main}

\end{document}